\documentclass[12pt,reqno]{amsart}
\usepackage{amsmath, amsthm, amscd, amsfonts, amssymb, graphicx, color}
\usepackage[bookmarksnumbered, colorlinks, plainpages]{hyperref}
\usepackage[cp1251]{inputenc}
\usepackage[T2A]{fontenc}
\usepackage[english]{babel}
\usepackage{amsmath,amsfonts,amssymb}
\usepackage{geometry}
\usepackage{cite}

\begin{document}
\title[An isomorphism theorem for parabolic problems]{An isomorphism theorem for parabolic problems in H\"ormander spaces and its applications}

\author[V.~Los, V.~Mikhailets, A.~Murach]
{Valerii Los, Vladimir A. Mikhailets, Aleksandr A. Murach}  

\address{Valerii Los \newline
National Technical University of Ukraine\ "Kyiv Polytechnic Institute";
\newline
Chernihiv National Technological University, Ukraine}
\email{v$\underline{\phantom{k}}$\,los@yahoo.com}

\address{Vladimir A. Mikhailets \newline
Institute of Mathematics, National Academy of Sciences of Ukraine;
\newline
National Technical University of Ukraine\ "Kyiv Polytechnic Institute"}
\email{mikhailets@imath.kiev.ua}

\address{Aleksandr A. Murach \newline
Institute of Mathematics, National Academy of Sciences of Ukraine;
\newline
Chernihiv National Pedagogical University, Ukraine}
\email{murach@imath.kiev.ua}

\subjclass[2000]{35K35, 46B70, 46E35}
\keywords{Parabolic initial-boundary value problem; H\"ormander space;
\hfill\break\indent slowly varying function; isomorphism property; interpolation with a function parameter}

\begin{abstract}
We investigate a general parabolic initial-boundary value problem with zero Cauchy data in some anisotropic H\"ormander inner product spaces. We prove that the operators corresponding to this problem are isomorphisms between appropriate H\"ormander spaces. As an application of this result, we establish a theorem on the local increase in regularity of solutions to the problem. We also obtain new sufficient conditions under which the generalized derivatives, of a given order, of the solutions should be continuous.
\end{abstract}

\maketitle
\numberwithin{equation}{section}
\newtheorem{theorem}{Theorem}[section]
\newtheorem{lemma}[theorem]{Lemma}
\newtheorem{definition}[theorem]{Definition}
\newtheorem{proposition}[theorem]{Proposition}
\newtheorem{remark}[theorem]{Remark}
\newtheorem{condition}[theorem]{Condition}
\newtheorem{corollary}[theorem]{Corollary}

\allowdisplaybreaks

\section{Introduction}\label{9sec1}

General linear parabolic initial-boundary value problems has been investigated completely enough on the classical scales of H\"older--Zygmund spaces and Sobolev spaces \cite{AgranovichVishik64, Eidelman69, Eidelman94, EidelmanZhitarashu98, Friedman64, LadyzhenskajaSolonnikovUraltzeva67, LionsMagenes72ii}. The central result of the theory of these problems states that they are well posed by Hadamard in appropriate pairs of the function spaces belonging to these scales; in other words, the operators corresponding to the parabolic problems are isomorphisms on these pairs. This fact has important applications to the investigation of the regularity of solutions to parabolic problems, to the study of properties of Green functions of these problems, to the optimal control problems and others.

However, for some applications to differential equations, the H\"older--Zygmund and Sobolev scales are not calibrated finely enough by number parameters \cite{Hormander63, Hormander83, MikhailetsMurach14}. In this connection, L.~H\"ormander \cite[Sect.~2.2]{Hormander63} introduced and investigated normed function spaces for which a function parameter, not a number, serves as an index of regularity of functions or distributions. L.~H\"ormander \cite{Hormander63, Hormander83} gave applications of these spaces to the investigation of regularity of solutions to hypoelliptic partial differential equations, among which parabolic equations are.

Of course, the inner product spaces are of the most important among H\"ormander spaces. Lately Mikhai\-lets and Murach \cite{MikhailetsMurach05UMJ5, MikhailetsMurach06UMJ2, MikhailetsMurach06UMJ3, MikhailetsMurach07UMJ5, MikhailetsMurach06UMJ11, MikhailetsMurach08UMJ4, MikhailetsMurach12BJMA2, MikhailetsMurach14, Murach07UMJ6} built the theory of general elliptic differential operators on manifolds and elliptic boundary-value problems in Hilbert scales formed by the H\"ormander spaces
\begin{equation}\label{9f1.1}
H^{s;\varphi}:=B_{2,\mu}:=\{w\in\mathcal{S}'(\mathbb{R}^{n}):\,
\mu\mathcal{F}w\in L_{2}(\mathbb{R}^{n})\}
\end{equation}
and their analogs for Euclidean domains and smooth manifolds. Here, the function
\begin{equation*}
\mu(\xi):=(1+|\xi|^{2})^{s/2}\varphi\bigl((1+|\xi|^{2})^{1/2}\bigr)
\quad\mbox{of}\quad\xi\in\mathbb{R}^{n}
\end{equation*}
serves as the index of regularity, with $s$ being a real number and with $\varphi$ being a slowly varying function at infinity in the sense of J.~Karamata. (As usual, $\mathcal{F}$ denotes the Fourier transform.) The class of the spaces \eqref{9f1.1} contains the
Sobolev scale $\{H^{s}\}=\{H^{s;1}\}$ and is attached to it by
means of $s$ but is calibrated more finely than the Sobolev scale.

This theory is based on the method of the interpolation with a function parameter between Hilbert spaces, specifically, between Sobolev spaces. Using this interpolation systematically, Mikhailets and Murach transferred the classical theory of elliptic equations from Sobolev spaces to wide classes of H\"ormander inner product spaces. Their results were supplemented in \cite{AnopMurach14MFAT2, AnopMurach14UMJ7, MurachZinchenko13MFAT1, ZinchenkoMurach12UMJ11} for the class of all Hilbert spaces that are interpolation spaces between Sobolev inner product spaces. (This class admits a description in terms of H\"ormander spaces and is closed with respect to the interpolation with a function parameter \cite{MikhailetsMurach13UMJ3, MikhailetsMurach15ResMath1}.)

Note that the methods of interpolation with a number parameter are fruitful in the classical theory of partial differential equations \cite{Berezansky68, LionsMagenes72i, LionsMagenes72ii, Triebel95}. However, to pass from the classical spaces parametrized by numbers parameters to more general classes of spaces parametrized by function parameters, we need to use the interpolation with a function parameter
\cite{CobosFernandez88, Merucci84, MikhailetsMurach08MFAT1}.

The purpose of this paper is to prove a theorem about an isomorphism generated by the general initial-boundary value parabolic problem in anisotropic analogs of the H\"ormander spaces \eqref{9f1.1}. We consider the problem in a bounded many-dimensional cylinder and suppose that the initial conditions are zero Cauchy data. (The case of inhomogeneous Cauchy data can be reduced to the homogeneous ones.) We will deduce this theorem from the classical result by M.~S.~Agranovich and M.~I.~Vishik \cite[Theorem 12.1]{AgranovichVishik64} by means of the interpolation with a function parameter between anisotropic Sobolev spaces. To this end, we investigate necessary interpolation properties of the used anisotropic spaces. We also give some applications of this isomorphism theorem to investigation of local regularity of solutions to the parabolic problem.

Some results of this paper are announced (without proofs) in~\cite{LosMurach14Dop6}. The case of a bounded two-dimensional cylinder was investigated in~\cite{LosMurach13MFAT2, Los15UMJ5}. In this case, the H\"ormander spaces over the lateral area of the cylinder are isotropic, which essentially simplifies the reasoning.

The paper consists of eight sections. Section~\ref{9sec1} is Introduction. In Section~\ref{9sec2}, we state an initial-boundary value problem for a general parabolic equation with zero Cauchy data. The necessary anisotropic H\"ormander spaces are introduced in Section~\ref{9sec3}. The next Section~\ref{9sec4} contains the main results of the paper. They are the isomorphism Theorem~\ref{9th4.1} and its applications, Theorems \ref{9th4.3} and \ref{9th4.4}. Theorem~\ref{9th4.3} says about the local regularity of generalized solutions to the parabolic problem in the H\"ormander spaces. Theorem \ref{9th4.4} gives sufficient conditions under which the generalized derivatives (of a prescribed order) of the solutions are continuous on a given subset of the cylinder. These conditions are formulated in terms of H\"ormander spaces. In Section~\ref{9sec5}, we discuss the isomorphism theorem in the case of anisotropic Sobolev spaces and compare it with the known result by M.~S.~Agranovich and M.~I.~Vishik.    Section~\ref{9sec6} is devoted to the method of the interpolation with a function parameter between Hilbert spaces. We recall its definition and some of its properties. In Section~\ref{9sec7}, we prove formulas that connect anisotropic Sobolev spaces with H\"ormander spaces by means of this interpolation. The last Section~\ref{9sec8} contains the proofs of the main results of this paper.

\section{Statement of the problem}\label{9sec2}

Let an integer $n\geq2$ and a real number $\tau>0$ be arbitrarily chosen. Suppose that $G$ is a bounded domain in $\mathbb{R}^{n}$ and that its boundary $\Gamma:=\partial G$ is an infinitely smooth closed manifold (of dimension $n-1$), with the $C^{\infty}$-structure on $\Gamma$ being induced by~$\mathbb{R}^{n}$. Let $\Omega:=G\times(0,\tau)$ and $S:=\Gamma\times(0,\tau)$. Thus, $\Omega$ is an open cylinder in $\mathbb{R}^{n+1}$, and $S$ is its lateral area, with their closures $\overline{\Omega}=\overline{G}\times[0,\tau]$ and
$\overline{S}=\Gamma\times[0,\tau]$.

We consider the following parabolic initial-boundary value problem in $\Omega$:
\begin{equation}\label{9f2.1}
\begin{gathered}
A(x,t,D_x,\partial_t)u(x,t)\equiv\sum_{|\alpha|+2b\beta\leq 2m}a^{\alpha,\beta}(x,t)\,D^\alpha_x\partial^\beta_t
u(x,t)=f(x,t)\\
\mbox{for all}\quad x\in G\quad\mbox{and}\quad t\in(0,\tau);
\end{gathered}
\end{equation}
\begin{equation}\label{9f2.2}
\begin{gathered}
B_{j}(x,t,D_x,\partial_t)u(x,t)\equiv\sum_{|\alpha|+2b\beta\leq m_j}
b_{j}^{\alpha,\beta}(x,t)\,D^\alpha_x\partial^\beta_t u(x,t)=g_{j}(x,t)\\
\mbox{for all}\quad x\in\Gamma,\quad t\in(0,\tau),\quad\mbox{and}\quad j\in\{1,\dots,m\};
\end{gathered}
\end{equation}
\begin{equation}\label{9f2.3}
\partial^k_t u(x,t)\big|_{t=0}=0\quad\mbox{for all}\quad x\in G\quad\mbox{and}\quad
k\in\{0,\ldots,\varkappa-1\}.
\end{equation}
Note that the initial data \eqref{9f2.3} are assumed to be zero. Here, $b$, $m$, and all $m_j$ are arbitrarily given integers such that $m\geq b\geq1$, $\varkappa:=m/b\in\mathbb{Z}$, and $m_j\geq0$. All coefficients of the linear partial differential expressions $A:=A(x,t,D_x,\partial_t)$ and $B_{j}:=B_{j}(x,t,D_x,\partial_t)$, with $j\in\{1,\dots,m\}$, are supposed to be infinitely smooth complex-valued functions given on $\overline{\Omega}$ and $\overline{S}$ respectively; i.e., each
\begin{equation*}
a^{\alpha,\beta}\in C^{\infty}(\overline{\Omega}):=
\bigl\{w\!\upharpoonright\overline{\Omega}\!:\,w\in C^{\infty}(\mathbb{R}^{n+1})\bigr\}
\end{equation*}
and each
\begin{equation*}
b_{j}^{\alpha,\beta}\in C^{\infty}(\overline{S}):=
\bigl\{v\!\upharpoonright\overline{S}\!:\,v\in C^{\infty}(\Gamma\times\mathbb{R})\bigr\}.
\end{equation*}
(Naturally, we consider $\Gamma\times\mathbb{R}$ as an infinitely smooth manifold with the $C^{\infty}$-structure induced by $\mathbb{R}^{n+1}$.)

We use the notation
$D^\alpha_x:=D^{\alpha_1}_{1}\dots D^{\alpha_n}_{n}$, with $D_{k}:=i\,\partial/\partial{x_k}$, and $\partial_t:=\partial/\partial t$
for partial derivatives of functions depending on $x=(x_1,\ldots,x_n)\in\mathbb{R}^{n}$ and $t\in\mathbb{R}$. Here, $i$ is imaginary unit, and $\alpha=(\alpha_1,...,\alpha_n)$ is a multi-index, with $|\alpha|:=\alpha_1+\cdots+\alpha_n$. In formulas \eqref{9f2.1} and \eqref{9f2.2} and their analogs, we take summation over the integer-valued nonnegative indices $\alpha_1,...,\alpha_n$ and $\beta$ that satisfy the condition written under the integral sign. As usual, $\xi^{\alpha}:=\xi_{1}^{\alpha_{1}}\ldots\xi_{n}^{\alpha_{n}}$ for $\xi:=(\xi_{1},\ldots,\xi_{n})\in\mathbb{C}^{n}$.

We recall \cite[Section~9, Subsection~1]{AgranovichVishik64} that the initial-boundary value problem \eqref{9f2.1}--\eqref{9f2.3} is said to be parabolic in $\Omega$ if the following Conditions \ref{9cond2.1} and \ref{9cond2.2} are fulfilled.

\begin{condition}\label{9cond2.1}
For arbitrary $x\in\overline{G}$, $t\in[0,\tau]$,
$\xi\in\mathbb{R}^{n}$, and $p\in\mathbb{C}$ with $\mathrm{Re}\,p\geq0$, we have
\begin{equation*}
A^{\circ}(x,t,\xi,p)\equiv\sum_{|\alpha|+2b\beta=2m} a^{\alpha,\beta}(x,t)\,\xi^\alpha
p^{\beta}\neq0\quad\mbox{whenever}\quad|\xi|+|p|\neq0.
\end{equation*}
\end{condition}

To formulate Condition~\ref{9cond2.2}, we arbitrarily choose a point $x\in\Gamma$, real number $t\in[0,\tau]$, vector $\xi\in\mathbb{R}^{n}$ tangent to the boundary $\Gamma$ at $x$, and number $p\in\mathbb{C}$ with $\mathrm{Re}\,p\geq0$ such that $|\xi|+|p|\neq0$. Let $\nu(x)$ is the unit vector of the inward normal to $\Gamma$ at $x$. It follows from Condition~\ref{9cond2.1} and the inequality $n\geq2$ that the polynomial
$A^{\circ}(x,t,\xi+\zeta\nu(x),p)$ in $\zeta\in\mathbb{C}$ has $m$ roots $\zeta^{+}_{j}(x,t,\xi,p)$, $j=\nobreak1,\ldots,m$, with positive imaginary part and $m$ roots with negative imaginary part provided that each root is taken the number of times equal to its multiplicity.

\begin{condition}\label{9cond2.2}
For each choice indicated above, the polynomials
$$
B_{j}^{\circ}(x,t,\xi+\zeta\nu(x),p)\equiv\sum_{|\alpha|+2b\beta=m_{j}}
b_{j}^{\alpha,\beta}(x,t)\,(\xi+\zeta\nu(x))^{\alpha}\,p^{\beta},\quad j=1,\dots,m,
$$
in $\zeta\in\mathbb{C}$ are linearly independent modulo
$$
\prod_{j=1}^{m}(\zeta-\zeta^{+}_{j}(x,t,\xi,p)).
$$
\end{condition}

Note that Condition~\ref{9cond2.1} is the condition for the partial differential equation $Au=\nobreak f$ to be $2b$-parabolic in $\overline{\Omega}$ in the sense of I.~G.~Petrovskii \cite{Petrovskii38}, whereas Condition~\ref{9cond2.2} claims  that the system of boundary partial differential expressions $\{B_{1},\ldots,B_{m}\}$ covers $A$ on $\overline{S}$.

We associate the linear mapping
\begin{gather}\label{9f2.4}
C^{\infty}_{+}(\overline{\Omega})\ni u\mapsto(Au,Bu)
:=\bigl(Au,B_1u,\ldots,B_mu\bigr)\in
C^{\infty}_{+}(\overline{\Omega})\times
\bigl(C^{\infty}_{+}(\overline{S})\bigr)^{m}
\end{gather}
with the parabolic problem \eqref{9f2.1}--\eqref{9f2.3}. Here and below, \begin{gather*}
C^{\infty}_{+}(\overline{\Omega}):=
\bigl\{w\!\upharpoonright\overline{\Omega}:\,
w\in C^{\infty}(\mathbb{R}^{n+1}),\;\,
\mathrm{supp}\,w\subseteq\mathbb{R}^{n}\times[0,\infty)\bigr\},\\
C^{\infty}_{+}(\overline{S}):=\bigl\{h\!\upharpoonright\overline{S}:
\,h\in C^{\infty}(\Gamma\times\mathbb{R}),\;\,\mathrm{supp}\,
h\subseteq\Gamma\times[0,\infty)\bigr\}.
\end{gather*}
In the paper, all functions and distributions are supposed to be complex-valued.

The mapping \eqref{9f2.4} is a bijection of $C^{\infty}_{+}(\overline{\Omega})$ onto
$C^{\infty}_{+}(\overline{\Omega})\times(C^{\infty}_{+}(\overline{S}))^{m}$. This follows specifically from \cite[Theorem 12.1]{AgranovichVishik64} (see the reasoning at the end of Section~\ref{9sec5}).

The main purpose of the paper is to prove that the mapping \eqref{9f2.4} extends uniquely (by continuity) to an isomorphism between appropriate H\"ormander inner product spaces.

\section{H\"ormander spaces}\label{9sec3}

Here, we will define the H\"ormander inner product spaces being used in the paper. They are built on the base of the anisotropic function spaces $H^{s,s/(2b);\varphi}(\mathbb{R}^{k+1})$ given over $\mathbb{R}^{k+1}$, with $k\geq1$. These spaces are parametrized with the pair of the real numbers $s$ and $s/(2b)$ and with the function $\varphi\in\mathcal{M}$.

By definition, the class $\mathcal{M}$ consists of all Borel measurable functions $\varphi:[1,\infty)\rightarrow(0,\infty)$ that satisfy the following two conditions:
\begin{itemize}
\item[(i)] both the functions $\varphi$ and $1/\varphi$ are bounded on each compact interval $[1,b]$ with $1<b<\infty$;
\item[(ii)] the function $\varphi$ varies slowly at infinity in the sense of J.~Karamata \cite{Karamata30a}, i.e.,
\begin{equation*}
\lim_{r\rightarrow\infty}\frac{\varphi(\lambda r)}{\varphi(r)}=1\quad\mbox{for every}\quad\lambda>0.
\end{equation*}
\end{itemize}

Note that the theory of slowly varying functions is set forth, e.g., in monographs \cite{BinghamGoldieTeugels89, Seneta76}. An important example of a function $\varphi\in\mathcal{M}$ is given by a continuous function $\varphi:[1,\infty)\rightarrow(0,\infty)$ such that
\begin{equation*}
\varphi(r)=(\log r)^{q_{1}}\,(\log\log r)^{q_{2}} \ldots
(\,\underbrace{\log\ldots\log}_{k\;\mbox{\small{times}}}r\,)^{q_{k}}
\quad\mbox{for}\quad r\gg1,
\end{equation*}
with $k\in\mathbb{Z}$, $k\geq1$, and    $q_{1},q_{2},\ldots,q_{k}\in\mathbb{R}$.

Let $s\in\mathbb{R}$, $\varphi\in\mathcal{M}$, and real $\gamma>0$.
Although we need the space $H^{s,s\gamma;\varphi}(\mathbb{R}^{k+1})$ only in the case where $\gamma=1/(2b)$, it is naturally to introduce this space for arbitrary $\gamma>0$.

By definition, the complex linear space $H^{s,s\gamma;\varphi}(\mathbb{R}^{k+1})$ consists of all tempered distributions $w$ on $\mathbb{R}^{k+1}$ whose (complete) Fourier transform $\widetilde{w}$ is locally Lebesgue integrable over $\mathbb{R}^{k+1}$ and meets the condition
\begin{equation*}
\int\limits_{\mathbb{R}^{k}}\int\limits_{\mathbb{R}}
r_{\gamma}^{2s}(\xi,\eta)\,\varphi^{2}(r_{\gamma}(\xi,\eta))\,
|\widetilde{w}(\xi,\eta)|^{2}\,d\xi\,d\eta<\infty.
\end{equation*}
Here and below, we use the notation
$$
r_{\gamma}(\xi,\eta):=\bigl(1+|\xi|^2+|\eta|^{2\gamma}\bigr)^{1/2},
\quad\mbox{with}
\quad\xi\in\mathbb{R}^{k}\quad\mbox{and}\quad\eta\in\mathbb{R}.
$$

This space is equipped with the inner product
\begin{equation*}
(w_{1},w_{2})_{H^{s,s\gamma;\varphi}(\mathbb{R}^{k+1})}:=
\int\limits_{\mathbb{R}^{k}}\int\limits_{\mathbb{R}}
r_{\gamma}^{2s}(\xi,\eta)\,\varphi^{2}(r_{\gamma}(\xi,\eta))\,
\widetilde{w_{1}}(\xi,\eta)\,\overline{\widetilde{w_{2}}(\xi,\eta)}\,
d\xi\,d\eta
\end{equation*}
of $w_{1},w_{2}\in H^{s,s\gamma;\varphi}(\mathbb{R}^{k+1})$. The inner product naturally induces the norm
\begin{equation*}
\|w\|_{H^{s,s\gamma;\varphi}(\mathbb{R}^{k+1})}:=
(w,w)_{H^{s,s\gamma;\varphi}(\mathbb{R}^{k+1})}^{1/2}.
\end{equation*}

We note that $H^{s,s\gamma;\varphi}(\mathbb{R}^{k+1})$ is a special case of the spaces $\mathcal{B}_{p,k}$ introduced by L.~H\"ormander \cite[Sect.~2.2]{Hormander63}. Namely, $H^{s,s\gamma;\varphi}(\mathbb{R}^{k+1})=\mathcal{B}_{p,k}$ provided that $p=2$ and the function parameter
\begin{equation*}
k(\xi,\eta)=r_{\gamma}^{s}(\xi,\eta)\,\varphi(r_{\gamma}(\xi,\eta))
\quad\mbox{for all}\quad
\xi\in\mathbb{R}^{k}\quad\mbox{and}\quad\eta\in\mathbb{R}.
\end{equation*}
The spaces $\mathcal{B}_{p,k}$ were systematically investigated by L.~H\"ormander \cite[Sect.~2.2]{Hormander63}, \cite[Sect.~10.1]{Hormander83} and in the $p=2$ case by L.~R.~Volevich and B.~P.~Paneah \cite{VolevichPaneah65}. If $\gamma=1/(2b)$, then we say that $H^{s,s\gamma;\varphi}(\mathbb{R}^{k+1})$ is a $2b$-anisotropic H\"ormander space.

According to \cite[Theorem 2.2.1]{Hormander63}, the H\"ormander space
$H^{s,s\gamma;\varphi}(\mathbb{R}^{k+1})$ is complete (i.e., Hilbert) and separable and is continuously embedded in the linear topological space $\mathcal{S}'(\mathbb{R}^{k+1})$ of tempered distributions on $\mathbb{R}^{k+1}$. Furthermore, the set $C^{\infty}_{0}(\mathbb{R}^{k+1})$ of test functions on $\mathbb{R}^{k+1}$ is dense in $H^{s,s\gamma;\varphi}(\mathbb{R}^{k+1})$.

In the $\varphi(r)\equiv1$ case, the space $H^{s,s\gamma;\varphi}(\mathbb{R}^{k+1})$ becomes an anisotropic Sobolev space and is denoted by $H^{s,s\gamma}(\mathbb{R}^{k+1})$. Generally, we have the continuous and dense embeddings
\begin{equation}\label{9f3.1}
\begin{gathered}
H^{s_{1},s_{1}\gamma}(\mathbb{R}^{k+1})\hookrightarrow
H^{s,s\gamma;\varphi}(\mathbb{R}^{k+1})\hookrightarrow
H^{s_{0},s_{0}\gamma}(\mathbb{R}^{k+1})\\
\mbox{whenever}\quad s_{0}<s<s_{1}.
\end{gathered}
\end{equation}
They result directly from the following property of $\varphi\in\mathcal{M}$: for every $\varepsilon>0$ there exists a number $c=c({\varepsilon})\geq1$ such that
$c^{-1}r^{-\varepsilon}\leq\varphi(r)\leq c\,r^{\varepsilon}$ for all $r\geq1$ (see \cite[Sect.~1.5, Subsect.~1]{Seneta76}).

The embeddings \eqref{9f3.1} clarify the role of the function parameter $\varphi\in\mathcal{M}$ in the class of Hilbert function spaces
\begin{equation}\label{9f3.2}
\bigl\{H^{s,s\gamma;\varphi}(\mathbb{R}^{k+1}):\,
s\in\mathbb{R},\,\varphi\in\mathcal{M}\,\bigr\}.
\end{equation}
We see that $\varphi$ defines a supplementary (subpower) regularity of distributions with respect to the basic (power) regularity given by the pair of numbers $(s,s\gamma)$. Specifically, if
$\varphi(r)\rightarrow\infty$ [or $\varphi(r)\rightarrow0$] as $r\rightarrow\infty$, then $\varphi$ defines a positive [or negative] supplementary regularity. So, we can briefly say that $\varphi$ refines the power anisotropic regularity $(s,s\gamma)$.

Note that in the $\gamma=1$ case the space $H^{s,s\gamma;\varphi}(\mathbb{R}^{k+1})$ becomes the isotropic H\"ormander space denoted by $H^{s;\varphi}(\mathbb{R}^{k+1})$. The spaces $H^{s;\varphi}(\mathbb{R}^{k+1})$, with $s\in\nobreak\mathbb{R}$ and $\varphi\in\mathcal{M}$, form the refined Sobolev scale introduced and investigated by V.~A.~Mikha\-i\-lets and A.~A.~Murach \cite{MikhailetsMurach05UMJ5, MikhailetsMurach06UMJ3}. This scale has various applications in the theory of elliptic partial differential equations \cite{MikhailetsMurach12BJMA2, MikhailetsMurach14}.

Using the scale \eqref{9f3.2}, we now introduce the function spaces relating to the problem \eqref{9f2.1}--\eqref{9f2.3}. Let $V$ be an open nonempty set in $\mathbb{R}^{k+1}$. (Specifically, $V=\Omega$, with $k=n$.) We put
\begin{equation}\label{9f3.3}
H^{s,s\gamma;\varphi}_{+}(V):=\bigl\{w\!\upharpoonright\!V:\,
w\in H^{s,s\gamma;\varphi}(\mathbb{R}^{k+1}),\;\,\mathrm{supp}\,
w\subseteq\mathbb{R}^{k}\times[0,\infty)\bigl\}.
\end{equation}
The norm in the linear space \eqref{9f3.3} is defined by the formula
\begin{equation}\label{9f3.4}
\begin{gathered}
\|u\|_{H^{s,s\gamma;\varphi}_{+}(V)}:=
\inf\bigl\{\,\|w\|_{H^{s,s\gamma;\varphi}(\mathbb{R}^{k+1})}:\\
w\in H^{s,s\gamma;\varphi}(\mathbb{R}^{k+1}),\;\,
\mathrm{supp}\,w\subseteq\mathbb{R}^{k}\times[0,\infty),\;\,
u=w\!\upharpoonright\!V\bigl\},
\end{gathered}
\end{equation}
with $u\in H^{s,s\gamma;\varphi}_{+}(V)$.

Considering the $V=\mathbb{R}^{k+1}$ case, we note that  $H^{s,s\gamma;\varphi}_{+}(\mathbb{R}^{k+1})$ consists of all $w\in H^{s,s\gamma;\varphi}(\mathbb{R}^{k+1})$ with $\mathrm{supp}\,w\subseteq\mathbb{R}^{k}\times[0,\infty)$ and is a (closed) subspace of the Hilbert space $H^{s,s\gamma;\varphi}(\mathbb{R}^{k+1})$. According to \cite[Lemma~3.3]{VolevichPaneah65}, the set
$$
C^{\infty}_{0}(\mathbb{R}^{k}\times(0,\infty)):=
\bigl\{w\in C^{\infty}_{0}(\mathbb{R}^{k+1}):\,
\mathrm{supp}\,w\subseteq\mathbb{R}^{k}\times(0,\infty)\bigr\}
$$
is dense in the space $H^{s,s\gamma;\varphi}_{+}(\mathbb{R}^{k+1})$.

Generally, $H^{s,s\gamma;\varphi}_{+}(V)$ is a Hilbert space because
formulas \eqref{9f3.3} and \eqref{9f3.4} mean that $H^{s,s\gamma;\varphi}_{+}(V)$ is the factor space of the Hilbert space $H^{s,s\gamma;\varphi}_{+}(\mathbb{R}^{k+1})$ by its subspace
\begin{equation}\label{9f3.5}
H^{s,s\gamma;\varphi}_{+}(\mathbb{R}^{k+1},V):=\bigl\{w\in H^{s,s\gamma;\varphi}_{+}(\mathbb{R}^{k+1}):\,
w=0\;\,\mbox{in}\;\,V\bigr\}.
\end{equation}
The norm \eqref{9f3.4} is induced by the inner product
$$
(u_{1},u_{2})_{H^{s,s\gamma;\varphi}_{+}(V)}:=
(w_{1}-\Upsilon w_{1},w_{2}-\Upsilon w_{2})_{H^{s,s\gamma;\varphi}(\mathbb{R}^{k+1})},
$$
with $u_{1},u_{2}\in H^{s,s\gamma;\varphi}_{+}(V)$. Here, $w_{j}\in H^{s,s\gamma;\varphi}_{+}(\mathbb{R}^{k+1})$, $w_{j}=u_{j}$ in $V$ for every $j\in\{1,2\}$, and $\Upsilon$ is the orthogonal projector of the space $H^{s,s\gamma;\varphi}_{+}(\mathbb{R}^{k+1})$ onto its subspace \eqref{9f3.5}.

In the Sobolev case of $\varphi(r)\equiv1$, we will omit the index $\varphi$ in the designations of $H^{s,s\gamma;\varphi}_{+}(V)$ and similar spaces. We note the continuous and dense embeddings
\begin{equation}\label{9f3.6}
H^{s_{1},s_{1}\gamma}_{+}(V)\hookrightarrow
H^{s,s\gamma;\varphi}_{+}(V)\hookrightarrow
H^{s_{0},s_{0}\gamma}_{+}(V)\quad
\mbox{whenever}\quad s_{0}<s<s_{1}.
\end{equation}
They result from \eqref{9f3.1} and the density of the set $\{w\!\upharpoonright\!V:w\in C^{\infty}_{0}(\mathbb{R}^{k}\times(0,\infty))\}$ in the spaces appearing in \eqref{9f3.6}.

As to the problem \eqref{9f2.1}--\eqref{9f2.3}, we need the space $H^{s,s\gamma;\varphi}_{+}(V)$ in the case where $k=n$ and $V=\Omega$. Note that the set $C^{\infty}_{+}(\overline{\Omega})$ is dense in $H^{s,s\gamma;\varphi}_{+}(\Omega)$.

We also need to introduce an analog of the space $H^{s,s\gamma;\varphi}_{+}(V)$ for the lateral area $S$ of the cylinder $\Omega$. It is sufficient for our purposes to restrict ourselves to the $s>0$ case. Let $\Pi:=\mathbb{R}^{n-1}\times(0,\tau)$, and consider the Hilbert space $H^{s,s\gamma;\varphi}_{+}(V)$ in the case where $k=n-1$ and $V=\Pi$. We will define the space $H^{s,s\gamma;\varphi}_{+}(S)$ on the base of the space $H^{s,s\gamma;\varphi}_{+}(\Pi)$ with the help of special local charts on $\overline{S}$.

We arbitrarily choose a finite atlas from the $C^{\infty}$-structure on the closed manifold $\Gamma$. Let this atlas be formed by the local charts $\theta_{j}:\mathbb{R}^{n-1}\leftrightarrow\Gamma_{j}$, with
$j=1,\ldots,\lambda$. Here, the open sets $\Gamma_{1},\ldots,\Gamma_{\lambda}$ make up a covering of $\Gamma$. Besides, we arbitrarily choose functions
$\chi_{j}\in C^{\infty}(\Gamma)$, with $j=1,\ldots,\lambda$, such that
$\mathrm{supp}\,\chi_{j}\subset\Gamma_{j}$ and $\chi_{1}+\cdots+\chi_{\lambda}=1$ on $\Gamma$. So, these functions form a $C^{\infty}$-partition of unity on $\Gamma$ which is subordinate to the covering.

This atlas of $\Gamma$ induces the collection of the special local charts
$$
\theta_{j}^{*}:\mathbb{R}^{n-1}\times[0,\tau]\leftrightarrow
\Gamma_{j}\times[0,\tau],\quad j=1,\ldots,\lambda,
$$
of $\overline{S}=\Gamma\times[0,\tau]$ by the formula $\theta_{j}^{*}(x,t):=(\theta_{j}(x),t)$ for all
$x\in\mathbb{R}^{n-1}$ and $t\in[0,\tau]$.
Consider the functions
$\chi_{j}^{*}(x,t):=\chi_{j}(x)$ of $x\in\Gamma$ and  $t\in[0,\tau]$, with $j=1,\ldots,\lambda$. They form a $C^{\infty}$-partition of unity on $\overline{S}$ which is subordinate to the covering $\{\Gamma_{j}\times[0,\tau]:j=1,\ldots,\lambda\}$ of~$\overline{S}$.

Now, we put
\begin{equation}\label{9f3.7}
H^{s,s\gamma;\varphi}_{+}(S):=\bigl\{v\in L_2(S):\,
(\chi_{j}^{*}v)\circ\theta_{j}^{*}\in H^{s,s\gamma;\varphi}_{+}(\Pi)
\;\,\mbox{for all}\;\,j\in\{1,\ldots,\lambda\}\bigr\}.
\end{equation}
Here, recall, $s>0$, and, $L_2(S)$ is the Hilbert space of all square integrable functions $v:S\rightarrow\mathbb{C}$ with respect to the Lebesgue measure on the smooth surface $S$. As usual, $\circ$ is the sign of composition of functions or mappings, so that
\begin{equation*}
((\chi_{j}^{*}v)\circ\theta_{j}^{*})(x,t)=
(\chi_{j}^{*}v)(\theta_{j}^{*}(x,t))=
\chi_{j}(\theta_{j}(x))\,v((\theta_{j}(x),t))
\end{equation*}
for all $x\in\mathbb{R}^{n-1}$ and $t\in(0,\tau)$. The inner product in the linear space \eqref{9f3.7} is defined by the formula
\begin{equation}\label{9f3.8}
(v_{1},v_{2})_{H^{s,s\gamma;\varphi}_{+}(S)}:=\sum_{j=1}^{\lambda}\,
((\chi_{j}^{*}v_{1})\circ\theta_{j}^{*},
(\chi_{j}^{*}v_{2})\circ
\theta_{j}^{*})_{H^{s,s\gamma;\varphi}_{+}(\Pi)},
\end{equation}
with $v_{1},v_{2}\in H^{s,s\gamma;\varphi}_{+}(S)$. This inner product naturally induces the norm
\begin{equation*}
\|v\|_{H^{s,s\gamma;\varphi}_{+}(S)}:=
(v,v)_{H^{s,s\gamma;\varphi}_{+}(S)}^{1/2}.
\end{equation*}

\begin{lemma}\label{9lem3.1}
Let $s>0$, $\gamma>0$, and $\varphi\in\mathcal{M}$. Then we state the following:
\begin{itemize}
\item[(i)] The space $H^{s,s\gamma;\varphi}_{+}(S)$ is complete (i.e., Hilbert) and separable and does not depend up to equivalence of norms on the indicated choice of an atlas and partition of unity on~$\Gamma$.
\item[(ii)] The set $C^{\infty}_{+}(\overline{S})$ is dense in this space.
\end{itemize}
\end{lemma}

We will prove this lemma at the end of Section~\ref{9sec7}.

Ending the present section, we note that statement \eqref{9f3.6} about continuous and dense embeddings also holds true in the case where $V=S$ and $s_{0}>0$. This follows directly from the validity of this statement for the open set $V=\Pi\subset\mathbb{R}^{n}$ and from Lemma~\ref{9lem3.1}(ii).

\section{Main results}\label{9sec4}

Here, we formulate an isomorphism theorem for the parabolic problem \eqref{9f2.1}--\eqref{9f2.3} in H\"ormander spaces introduced above and then consider applications of this theorem to the investigation of the regularity of the generalized solutions to the problem.

Let $\sigma_0$ denote the smallest integer such that
$$
\sigma_0\geq2m,\quad\sigma_0\geq m_j+1\;\;\mbox{for each}\;\;j\in\{1,\ldots,m\},
\quad\mbox{and}\quad\frac{\sigma_0}{2b}\in\mathbb{Z}.
$$
Note, if $m_j\leq2m-1$ for each $j\in\{1,\ldots,m\}$, then $\sigma_0=2m$.

Isomorphism Theorem is formulated as follows.

\begin{theorem}\label{9th4.1}
For every real number $\sigma>\sigma_0$ and every function parameter $\varphi\in\nobreak\mathcal{M}$, the mapping \eqref{9f2.4} extends uniquely (by continuity) to an isomorphism
\begin{equation}\label{9f4.1}
(A,B):H^{\sigma,\sigma/(2b);\varphi}_{+}(\Omega)\leftrightarrow
\mathcal{H}^{\sigma-2m,(\sigma-2m)/(2b);\varphi}_{+}(\Omega,S),
\end{equation}
where
\begin{equation}\label{9f4.2}
\begin{gathered}
\mathcal{H}^{\sigma-2m,(\sigma-2m)/(2b);\varphi}_{+}(\Omega,S)\\
:=H^{\sigma-2m,(\sigma-2m)/(2b);\varphi}_{+}(\Omega)\oplus
\bigoplus_{j=1}^{m}H^{\sigma-m_j-1/2,(\sigma-m_j-1/2)/(2b);\varphi}_{+}(S).
\end{gathered}
\end{equation}
\end{theorem}

In the Sobolev case of $\varphi(r)\equiv1$ and $\sigma/(2b)\in\mathbb{Z}$, this theorem follows from the result by M.~S.~Agranovich and M.~I.~Vishik \cite[Theorem 12.1]{AgranovichVishik64}, the limiting case of $\sigma=\sigma_0$ being included. This will be demonstrated in Section~\ref{9sec5}. In the general situation, we will deduce Theorem~\ref{9th4.1} from the Sobolev case with the help of the interpolation with a function parameter between Hilbert spaces. This will be done in Section~\ref{9sec8} after we investigate the necessary interpolation properties of the H\"ormander spaces appearing in \eqref{9f4.1} and \eqref{9f4.2}.

As has just been mentioned, the mapping \eqref{9f2.4} extends by continuity to an isomorphism
\begin{equation}\label{9f4.3}
(A,B):H^{\sigma_{0},\sigma_{0}/(2b)}_{+}(\Omega)\leftrightarrow
\mathcal{H}^{\sigma_{0}-2m,(\sigma_{0}-2m)/(2b)}_{+}(\Omega,S),
\end{equation}
which acts between some anisotropic Sobolev spaces. All the isomorphisms
\eqref{9f4.1}, with $\sigma>\sigma_0$ and $\varphi\in\nobreak\mathcal{M}$, are restrictions of \eqref{9f4.3}. This results from the embeddings \eqref{9f3.6} being valid for $V=\Omega$ and $V=S$.

Every vector
\begin{equation}\label{9f4.4}
(f,g_{1},...,g_{m})\in
\mathcal{H}^{\sigma_{0}-2m,(\sigma_{0}-2m)/(2b)}_{+}(\Omega,S)
\end{equation}
has a unique preimage $u\in H^{\sigma_{0},\sigma_{0}/(2b)}_{+}(\Omega)$ relative to the one-to-one mapping \eqref{9f4.3}. The function $u$ is said to be a (strong) generalized solution to the parabolic problem \eqref{9f2.1}--\eqref{9f2.3} with the right-hand sides \eqref{9f4.4}.

Let us now discuss the regularity properties of this solution in H\"ormander spaces. The next result follows immediately from Theorem~\ref{9th4.1}.

\begin{corollary}\label{9cor4.2}
Assume that  $u\in H^{\sigma_{0},\sigma_{0}/(2b)}_{+}(\Omega)$ is a generalized solution to the problem \eqref{9f2.1}--\eqref{9f2.3} whose right-hand sides satisfy the condition
\begin{equation*}
(f,g_{1},...,g_{m})\in
\mathcal{H}^{\sigma-2m,(\sigma-2m)/(2b);\varphi}_{+}(\Omega,S)
\end{equation*}
for some $\sigma>\sigma_{0}$ and $\varphi\in\nobreak\mathcal{M}$. Then
$u\in H^{\sigma,\sigma/(2b);\varphi}_{+}(\Omega)$.
\end{corollary}

We observe that the supplementary regularity $\varphi$ of the right-hand sides is inherited by the solution.

Let us formulate a local version of this result. Let $U$ be an open set in $\mathbb{R}^{n+1}$, and let $\omega:=U\cap\Omega\neq\varnothing$,
$\pi_{1}:=U\cap\partial\Omega$, and $\pi_{2}:=U\cap S$. We need to introduce local analogs of the spaces $H^{s,s\gamma;\varphi}_{+}(\Omega)$ and $H^{s,s\gamma;\varphi}_{+}(S)$ with $s>0$, $\gamma=1/(2b)$, and $\varphi\in\mathcal{M}$.

We let $H^{s,s\gamma;\varphi}_{+,\mathrm{loc}}(\omega,\pi_1)$ denote the linear space of all distributions $u$ on $\Omega$ such that $\chi u\in\nobreak H^{s,s\gamma;\varphi}_{+}(\Omega)$ for each function $\chi\in C^\infty (\overline\Omega)$ with $\mathrm{supp}\,\chi\subset\omega\cup\pi_1$. The topology in this space is given by the seminorms
$$
u\mapsto\|\chi u\|_{H^{s,s\gamma;\varphi}_{+}(\Omega)},
$$
where $\chi$ is an arbitrary above-mentioned function. Analogously, we let $H^{s,s\gamma;\varphi}_{+,\mathrm{loc}}(\pi_2)$ denote the linear space of all distributions $v$ on $S$ such that $\chi v\in H^{s,s\gamma;\varphi}_{+}(S)$ for each function $\chi\in C^{\infty}(\overline S)$ with $\mathrm{supp}\,\chi\subset\pi_2$. The topology in this space is given by the seminorms
$$
v\mapsto\|\chi v\|_{H^{s,s\gamma;\varphi}_{+}(S)},
$$
where $\chi$ is an arbitrary function mentioned just now.

\begin{theorem}\label{9th4.3}
Let $u\in H^{\sigma_0,\sigma_0/(2b)}_+(\Omega)$ be a generalized solution to the parabolic problem \eqref{9f2.1}--\eqref{9f2.3} with the right-hand sides \eqref{9f4.4}. Assume that
\begin{gather}\label{9f4.5}
f\in H^{\sigma-2m,(\sigma-2m)/(2b);\varphi}_{+,\mathrm{loc}}
(\omega,\pi_1),\\
g_{j}\in H^{\sigma-m_j-1/2,(\sigma-m_j-1/2)/(2b);\varphi}_{+,\mathrm{loc}}
(\pi_2),\quad\mbox{with}\quad j=1,\dots,m,
\label{9f4.6}
\end{gather}
for some $\sigma>\sigma_0$ and $\varphi\in\mathcal{M}$. Then $u\in H^{\sigma,\sigma/(2b);\varphi}_{+,\mathrm{loc}}(\omega,\pi_1)$.
\end{theorem}

In the special case where $\omega=\Omega$ and $\pi_1=\partial\Omega$ (then $\pi_2=S$), Theorem~\ref{9th4.3} says about the global increasing in smoothness, so that we arrive at Corollary~\ref{9cor4.2}. If $\pi_1=\varnothing$, then this theorem asserts that the regularity increases in neighbourhoods of interior points of the closed domain~$\overline{\Omega}$.

Using H\"ormander spaces, we can obtain fine sufficient conditions under which the generalized solution $u$ and its generalized derivatives of a prescribed order are continuous on $\omega\cup\pi_1$.

\begin{theorem}\label{9th4.4}
Let an integer $p\geq0$ be such that $p+b+n/2>\sigma_{0}$, and let $u\in H^{\sigma_0,\sigma_0/(2b)}_+(\Omega)$ be a generalized solution to the parabolic problem \eqref{9f2.1}--\eqref{9f2.3} with the right-hand sides \eqref{9f4.4}. Suppose that they satisfy conditions \eqref{9f4.5}, \eqref{9f4.6} for $\sigma:=p+b+n/2$ and some function parameter $\varphi\in\mathcal{M}$ subject to
\begin{equation}\label{9f4.7}
\int\limits_{1}^{\infty}\;\frac{dr}{r\varphi^2(r)}<\infty.
\end{equation}
Then the solution $u(x,t)$ and all its generalized derivatives
$D_{x}^{\alpha}\partial_{t}^{\beta}u(x,t)$ with $|\alpha|+2b\beta\leq p$ are continuous on $\omega\cup\pi_1$.
\end{theorem}

\begin{remark}\label{9rem4.5}\rm
Condition \eqref{9f4.7} in Theorem~\ref{9th4.4} is sharp. If each solution $u\in\nobreak H^{\sigma_0,\sigma_0/(2b)}_+(\Omega)$ to the problem \eqref{9f2.1}--\eqref{9f2.3} whose right-hand sides satisfy conditions
\eqref{9f4.5}, \eqref{9f4.6} for $\sigma:=p+b+n/2$ and given $\varphi\in\mathcal{M}$ complies with the conclusion of Theorem~\ref{9th4.4}, then $\varphi$ meets condition \eqref{9f4.7}.
\end{remark}

\begin{remark}\label{9rem4.6}\rm
If we formulate an analog of Theorem~\ref{9th4.4} for the Sobolev case of $\varphi\equiv\nobreak1$, we have to change the condition of this theorem for a stronger one. Namely, we have to claim that the right-hand sides of the problem \eqref{9f2.1}--\eqref{9f2.3} satisfy conditions \eqref{9f4.5}, \eqref{9f4.6} for certain $\sigma>p+b+n/2$.
\end{remark}

In Section~\ref{9sec8}, we will deduce Theorem \ref{9th4.3} from Isomorphism Theorem~\ref{9th4.1} and will show that Theorem~\ref{9th4.4} is a consequence of Theorem~\ref{9th4.3} and a version of H\"ormander's embedding theorem \cite[Theorem 2.2.7]{Hormander63}. We will also justify Remark~\ref{9rem4.5}.

\section{Isomorphism Theorem in the Sobolev case}\label{9sec5}

The goal of this section is to show that Theorem \ref{9th4.1} follows from the above-mentioned result by M.~S.~Agranovich and M.~I.~Vishik \cite[Theorem 12.1]{AgranovichVishik64} in the Sobolev case where $\varphi(r)\equiv1$ and $\sigma/(2b)\in\mathbb{Z}$. Beforehand, we will prove a lemma about a description of the spaces
$H^{s,s\gamma}_{+}(\Omega)$ and $H^{s,s\gamma}_{+}(S)$ in terms of the Hilbert spaces $H^{s,s\gamma}(\Omega)$ and $H^{s,s\gamma}(S)$ used in this result. The latter two are defined quite similarly to the first two, with no restrictions on support of distributions being imposed. For the reader's convenience, we give the relevant definitions.

Let real $s>0$ and  $\gamma>0$. Suppose that $V$ is an open nonempty set in $\mathbb{R}^{k+1}$, with $k\geq1$. We put
\begin{equation}\label{9f5.1}
H^{s,s\gamma}(V):=\bigl\{w\!\upharpoonright\!V:\,
w\in H^{s,s\gamma}(\mathbb{R}^{k+1})\bigl\}.
\end{equation}
The norm in the linear space \eqref{9f5.1} is defined by the formula
\begin{equation}\label{9f5.2}
\|u\|_{H^{s,s\gamma}(V)}:=
\inf\bigl\{\,\|w\|_{H^{s,s\gamma}(\mathbb{R}^{k+1})}:\\
w\in H^{s,s\gamma}(\mathbb{R}^{k+1}),\;\,
u=w\!\upharpoonright\!V\bigl\},
\end{equation}
with $u\in H^{s,s\gamma}(V)$. This space is Hilbert with respect to the norm \eqref{9f5.2}. We are interested in the case where $V=\Omega$, with $k=n$, and also in the case where $V=\Pi$, with $\Pi:=\mathbb{R}^{n-1}\times(0,\tau)$ and $k=n-1$. Using $H^{s,s\gamma}(\Pi)$, we now define the Hilbert space $H^{s,s\gamma}(S)$ by formulas \eqref{9f3.7}, \eqref{9f3.8} in which we omit the subscript~$+$. The anisotropic Sobolev spaces just defined are well known in the theory of parabolic equations \cite{AgranovichVishik64, EidelmanZhitarashu98, LionsMagenes72ii, Slobodeckii58}.

We have the continuous embeddings
\begin{equation}\label{9f5.3}
H^{s,s\gamma}_{+}(\Omega)\hookrightarrow H^{s,s\gamma}(\Omega)
\quad\mbox{and}\quad
H^{s,s\gamma}_{+}(S)\hookrightarrow H^{s,s\gamma}(S).
\end{equation}
They follow immediately from the definitions of the spaces appearing in~\eqref{9f5.3}.

Besides, we let $H^{\theta}(U)$ denote the isotropic Sobolev inner product space of order $\theta$ over a Euclidean domain~$U$; specifically, $H^{0}(U)=L_{2}(U)$.

We need the following version of the result by M.~S.~Agranovich and M.~I.~Vishik \cite[Proposition~8.1]{AgranovichVishik64}.

\begin{lemma}\label{9lem5.1}
Let $s>0$, $\gamma>0$, and $s\gamma-1/2\notin\mathbb{Z}$. Then the space
$H^{s,s\gamma}_{+}(\Omega)$ consists of all functions $u\in H^{s,s\gamma}(\Omega)$ such that
\begin{equation}\label{9f5.4}
\begin{aligned}
\partial^k_t u(x,t)\big|_{t=0}=0\quad&\mbox{for almost all}\quad x\in G\\
&\mbox{whenever}\quad k\in\mathbb{Z}
\quad\mbox{and}\quad 0\leq k<s\gamma-1/2.
\end{aligned}
\end{equation}
Moreover, the norm in $H^{s,s\gamma}_{+}(\Omega)$ is equivalent to the norm in $H^{s,s\gamma}(\Omega)$. This lemma remains valid if we replace $\Omega$ by $S$ and $G$ by $\Gamma$.
\end{lemma}

Lemma~\ref{9lem5.1} is close to the result by M.~S.~Agranovich and M.~I.~Vishik \cite[Proposition~8.1]{AgranovichVishik64}. They found necessary and sufficient conditions under which the extension of every  function $u\in H^{s,s\gamma}(G\times(0,\theta))$ by zero belongs to $H^{s,s\gamma}(G\times(-\infty,\theta))$ with $0<\theta\leq\infty$, they restricting themselves to the case where $s\in\mathbb{Z}$ and $\gamma=1/(2b)$. These conditions are tantamount to \eqref{9f5.4} and imply equivalence of the norms of $u$ and its extension by zero. M.~S.~Agranovich and M.~I.~Vishik also considered the case of functions given on $\Gamma\times(0,\theta)$

\begin{proof}[Proof of Lemma $\ref{9lem5.1}$.]
We note first that condition \eqref{9f5.4} is well posed by virtue of the trace theorem for anisotropic Sobolev spaces (see, e.g., \cite[Part~II, Theorem~4]{Slobodeckii58}). Let $\Upsilon^{s,s\gamma}(\Omega)$ denote the linear manifold of all functions $u\in H^{s,s\gamma}(\Omega)$ that satisfy~\eqref{9f5.4}. According to this trace theorem, we may and will consider $\Upsilon^{s,s\gamma}(\Omega)$ as a (closed) subspace of $H^{s,s\gamma}(\Omega)$. It follows directly from \eqref{9f5.3} that we have the continuous embedding $H^{s,s\gamma}_{+}(\Omega)\hookrightarrow \Upsilon^{s,s\gamma}(\Omega)$. Therefore (in view of the Banach theorem on inverse operator) it remains to prove the converse inclusion $\Upsilon^{s,s\gamma}(\Omega)\subset H^{s,s\gamma}_{+}(\Omega)$.

Let $u\in\Upsilon^{s,s\gamma}(\Omega)$. We must prove that $u=w$ on $\Omega$ for a certain function $w\in H^{s,s\gamma}_{+}(\mathbb{R}^{n+1})$. To this end, we use the following three extension operators $O$, $T_{\tau}$, and $T_{G}$ acting between some isotropic Sobolev spaces.

Given a function $v\in L_{2}((0,\infty))$, we define the function $Ov\in L_{2}(\mathbb{R})$ by the formulas $(Ov)(t):=v(t)$ for $t>0$ and $(Ov)(t):=0$ for $t\leq0$. Thus, we introduce a bounded linear operator
\begin{equation}\label{9f5.5}
O:L_{2}((0,\infty))\to L_{2}(\mathbb{R}).
\end{equation}
Let $\Upsilon^{s\gamma}((0,\infty))$ denote the linear manifold of all functions $v\in H^{s\gamma}((0,\infty)$ such that $v^{(k)}(0)=\nobreak0$ whenever $k\in\mathbb{Z}$ satisfies $0\leq k<s\gamma-1/2$. By the trace theorem, $\Upsilon^{s\gamma}((0,\infty))$ is a subspace of
$H^{s\gamma}((0,\infty)$. It follows directly from \cite[Theorems 2.9.3(a) and 2.10.3(b)]{Triebel95} and the condition $s\gamma-1/2\notin\mathbb{Z}$ that the restriction of \eqref{9f5.5} to $H^{s\gamma}((0,\infty))$ defines an isomorphism
\begin{equation}\label{9f5.6}
O:\Upsilon^{s\gamma}((0,\infty))\leftrightarrow H^{s\gamma}_{+}(\mathbb{R}).
\end{equation}
Here, $H^{s\gamma}_{+}(\mathbb{R})$ consists, by definition, of all functions $v\in H^{s\gamma}(\mathbb{R})$ with $\mathrm{supp}\,v\subseteq[0,\infty)$ and is regarded as a subspace of $H^{s\gamma}(\mathbb{R})$.

We consider a bounded linear operator
\begin{equation}\label{9f5.7}
T_{\tau}:L_{2}((0,\tau))\rightarrow L_{2}((0,\infty))
\end{equation}
such that $T_{\tau}v=v$ on $(0,\tau)$ for every function $v\in L_{2}((0,\tau))$ and that
the restriction of the mapping $T_{\tau}$ to the space $H^{s\gamma}((0,\tau))$ is a bounded operator
\begin{equation}\label{9f5.8}
T_{\tau}:H^{s\gamma}((0,\tau))\rightarrow H^{s\gamma}((0,\infty)).
\end{equation}
We also consider a bounded linear operator
\begin{equation}\label{9f5.9}
T_{G}:L_{2}(G)\rightarrow L_{2}(\mathbb{R}^{n})
\end{equation}
such that $T_{G}h=h$ on $G$ for every function $h\in L_{2}(G)$ and that
the restriction of the mapping $T_{G}$ to the space $H^{s}(G)$ is a bounded operator
\begin{equation}\label{9f5.10}
T_{G}:H^{s}(G)\rightarrow H^{s}(\mathbb{R}^{n}).
\end{equation}
Operators of this kind exist \cite[Theorems 4.2.2 and 4.2.3]{Triebel95}.

It is known that
\begin{equation}\label{9f5.11}
H^{s,s\gamma}(\Omega)=H^{s}(G)\otimes L_{2}((0,\tau))\cap
L_{2}(G)\otimes H^{s\gamma}((0,\tau))
\end{equation}
and
\begin{equation}\label{9f5.12}
H^{s,s\gamma}(\mathbb{R}^{n+1})=
H^{s}(\mathbb{R}^{n})\otimes L_{2}(\mathbb{R})\cap
L_{2}(\mathbb{R}^{n})\otimes H^{s\gamma}(\mathbb{R})
\end{equation}
up to equivalence of norms; see, e.g., \cite[\S~8, Subection~1]{AgranovichVishik64}. (As usual, $E\otimes F$ denotes the tensor product of arbitrary Hilbert spaces $E$ and $F$. Besides, their intersection $E\cap F$ is considered as a Hilbert space endowed with the inner product $(v_{1},v_{2})_{E\cap F}:=
(v_{1},v_{2})_{E}+(v_{1},v_{2})_{F}$ of vectors $v_{1},v_{2}\in E\cap F$.)

It follows directly from \eqref{9f5.7}, \eqref{9f5.8}, \eqref{9f5.11} and the inclusion $u\in\Upsilon^{s,s\gamma}(\Omega)$ that
\begin{equation*}
(I\otimes T_{\tau})u\in H^{s}(G)\otimes L_{2}((0,\infty))\cap
L_{2}(G)\otimes \Upsilon^{s\gamma}((0,\infty)).
\end{equation*}
Here, $I$ is the identity operator on $L_{2}(G)$. Then
\begin{equation}\label{9f5.13}
\begin{aligned}
w&:=(T_{G}\otimes(OT_{\tau}))u=(T_{G}\otimes O)(I\otimes T_{\tau})w\\
&\in H^{s}(\mathbb{R}^{n})\otimes L_{2}(\mathbb{R})\cap L_{2}(\mathbb{R}^{n})\otimes H^{s\gamma}_{+}(\mathbb{R})=
H^{s,s\gamma}_{+}(\mathbb{R}^{n+1})
\end{aligned}
\end{equation}
in view of formulas \eqref{9f5.5}, \eqref{9f5.6}, \eqref{9f5.9}, \eqref{9f5.10}, and \eqref{9f5.12}. Besides, $w=u$ on $\Omega$. Thus, $u\in H^{s,s\gamma}_{+}(\Omega)$.

The same reasoning shows that Lemma~\ref{9lem5.1} remains valid if we replace $\Omega$ by $\Pi:=\mathbb{R}^{n-1}\times(0,\tau)$ and $G$ by $\mathbb{R}^{n-1}$. (Of course, we take $\mathbb{R}^{n}$ instead of $\mathbb{R}^{n+1}$ and need not use the extension operator $T_{G}$ in this case.) It follows directly from this fact and the definitions of $H^{s,s\gamma}_{+}(S)$ and $H^{s,s\gamma}(S)$ that Lemma~\ref{9lem5.1} also remains valid if we replace $\Omega$ by $S$ and $G$ by~$\Gamma$.
\end{proof}

Discussing Theorem \ref{9th4.1} in this section, we restrict ourselves to the Sobolev case where $\varphi(r)\equiv1$ and $\sigma/(2b)\in\mathbb{Z}$ and suppose that $\sigma\geq\sigma_{0}$. Let us show that Theorem \ref{9th4.1} in this case follows from M.~S.~Agranovich and M.~I.~Vishik's result \cite[Theorem 12.1]{AgranovichVishik64}.

We consider the parabolic initial-boundary value problem \eqref{9f2.1}--\eqref{9f2.3} for
arbitrarily chosen right-hand sides
\begin{equation}\label{9f5.14}
(f,g_1,\ldots,g_m)\in
\mathcal{H}_{+}^{\sigma-2m,(\sigma-2m)/(2b)}(\Omega,S).
\end{equation}
Of course, the equalities and partial derivatives appearing in this problem are interpreted in the sense of the theory of distributions. The vector \eqref{9f5.14} satisfies the compatibility condition \cite[\S~11]{AgranovichVishik64} in the case of zero initial data \eqref{9f2.3}. M.~S.~Agranovich and M.~I.~Vishik's theorem \cite[Theorem~12.1]{AgranovichVishik64} asserts in view of \eqref{9f5.3} that the problem \eqref{9f2.1}--\eqref{9f2.3} has a unique solution $u\in H^{\sigma,\sigma/(2b)}(\Omega)$ and that this solution satisfies the two-sided estimate
\begin{equation}\label{9f5.15}
\begin{aligned}
\|u\|_{H^{\sigma,\sigma/(2b)}(\Omega)}&\leq c_1
\|(f,g_1,\ldots,g_m)\|_
{\mathcal{H}^{\sigma-2m,(\sigma-2m)/(2b)}(\Omega,S)}\\
&\leq c_2\,\|u\|_{H^{\sigma,\sigma/(2b)}(\Omega)}
\end{aligned}
\end{equation}
with some positive numbers $c_1$ and $c_2$ being independent of \eqref{9f5.14} and $u$. Here, we put
\begin{gather*}
\mathcal{H}^{\sigma-2m,(\sigma-2m)/(2b)}(\Omega,S)\\
:=H^{\sigma-2m,(\sigma-2m)/(2b)}(\Omega)\oplus
\bigoplus_{j=1}^{m}H^{\sigma-m_j-1/2,(\sigma-m_j-1/2)/(2b)}(S).
\end{gather*}

Let us show that $u\in H^{\sigma,\sigma/(2b)}_{+}(\Omega)$. To this end, we use Lemma~\ref{9lem5.1} with $s:=\sigma$ and $\gamma:=1/(2b)$. According to \eqref{9f2.3}, the function $u$ satisfies condition \eqref{9f5.4} provided that $0\leq k\leq\varkappa-1$. Here,
\begin{equation*}
\varkappa-1=\frac{2m}{2b}-1<\frac{\sigma}{2b}-\frac{1}{2}.
\end{equation*}
The fulfilment of \eqref{9f5.4} for the rest values of the integer $k$ (when $\varkappa-1<k<\sigma/(2b)-1/2$) is proved (if these values exist) in the following way.

Let the number of these values is $l\geq1$; then
\begin{equation}\label{9f5.16}
\varkappa+l-1<\frac{\sigma}{2b}-\frac{1}{2}<\varkappa+l.
\end{equation}
The parabolicity Condition \ref{9cond2.1} in the case of $\xi=0$ and $p=1$ means that the coefficient $a^{(0,\ldots,0),\varkappa}(x,t)\neq0$ for all $x\in\overline{G}$ and $t\in[0,\tau]$. Therefore we can resolve the parabolic equation \eqref{9f2.1} with respect to $\partial^\varkappa_t
u(x,t)$; namely, we can write
\begin{equation}\label{9f5.17}
\partial^\varkappa_t u(x,t)=
\sum_{\substack{|\alpha|+2b\beta\leq 2m,\\ \beta\leq\varkappa-1}}
a_{0}^{\alpha,\beta}(x,t)\,D^\alpha_x\partial^\beta_t
u(x,t)+f(x,t)
\end{equation}
for some functions $a_{0}^{\alpha,\beta}\in C^{\infty}(\overline{\Omega})$. If $l\geq2$, then we differentiate the equality \eqref{9f5.17} $l-1$ times with respect to $t$ and obtain $l-1$ equalities
\begin{equation}\label{9f5.18}
\begin{gathered}
\partial^{\varkappa+j}_{t}u(x,t)=
\sum_{\substack{|\alpha|+2b\beta\leq 2m+2bj,\\
|\alpha|\leq2m,\;\beta\leq\varkappa+j-1}}
a_{j}^{\alpha,\beta}(x,t)\,D^\alpha_x\partial^\beta_t u(x,t)+ \partial^{j}_{t}f(x,t),\\
\mbox{with}\quad j=1,\ldots,l-1.
\end{gathered}
\end{equation}
Here, each $a_{j}^{\alpha,\beta}(x,t)$ is a certain function from $C^{\infty}(\overline{\Omega})$. The equalities \eqref{9f5.17} and
\eqref{9f5.18} are considered on $\Omega=\{(x,t):x\in G,0<t<\tau\}$.
Since $f\in H_{+}^{\sigma-2m,(\sigma-2m)/(2b)}(\Omega)$, then
$\partial^{j}_{t}f(x,t)\big|_{t=0}=0$ for almost all $x\in G$ and for each integer $j\in\{0,\ldots,l-1\}$ by virtue of Lemma~\ref{9lem5.1} and \eqref{9f5.16}. Using these equalities, we deduce successively from \eqref{9f5.17} and \eqref{9f5.18} that $\partial^{\varkappa+j}_{t}u(x,t)\big|_{t=0}=0$ for the same $x$ and $j$.

Thus, the function $u\in H^{\sigma,\sigma/(2b)}(\Omega)$ satisfies condition \eqref{9f5.4}. Therefore $u\in H_{+}^{\sigma,\sigma/(2b)}(\Omega)$ due to Lemma~\ref{9lem5.1}. Moreover, according to this lemma and formulas \eqref{9f5.14} and \eqref{9f5.15}, we can write
\begin{equation}\label{9f5.19}
\begin{aligned}
\|u\|_{H_{+}^{\sigma,\sigma/(2b)}(\Omega)}&\leq c_3
\|(f,g_1,\ldots,g_m)\|_
{\mathcal{H}_{+}^{\sigma-2m,(\sigma-2m)/(2b)}(\Omega,S)}\\
&\leq c_4\,\|u\|_{H_{+}^{\sigma,\sigma/(2b)}(\Omega)}.
\end{aligned}
\end{equation}
Here, $c_3$ and $c_4$ are some positive numbers that do not depend on \eqref{9f5.14} and~$u$.

Thus, we conclude that for an arbitrary vector \eqref{9f5.14} there exists a unique solution $u\in H_{+}^{\sigma,\sigma/(2b)}(\Omega)$ of the parabolic problem \eqref{9f2.1}--\eqref{9f2.3} and that this solution satisfies~\eqref{9f5.19}. Evidently, this conclusion is equivalent to Theorem \ref{9th4.1} in the Sobolev case under consideration. Therefore Theorem \ref{9th4.1} in this case is a consequence of M.~S.~Agranovich and M.~I.~Vishik's result \cite[Theorem 12.1]{AgranovichVishik64}.

Completing this section, we will show that the mapping \eqref{9f2.4} is   a bijection
\begin{equation}\label{9f5.20}
(A,B):C^{\infty}_{+}(\overline{\Omega})\leftrightarrow
C^{\infty}_{+}(\overline{\Omega})\times
\bigl(C^{\infty}_{+}(\overline{S})\bigr)^{m},
\end{equation}
as we asserted at the end of Section~\ref{9sec2}. Let us use Theorem \ref{9th4.1} in the Sobolev case just considered, namely, where $\sigma=2bl>\sigma_{0}$ with $l\in\mathbb{Z}$. The mapping \eqref{9f2.4} is injective by this theorem. It remains to show that \eqref{9f2.4} is surjective. Let a vector $(f,g_1,\ldots,g_m)$ satisfy~\eqref{9f5.14}. According to this theorem, there exists a unique function
\begin{equation}\label{9f5.21}
u\in\bigcap_{l\in\mathbb{Z},\,2bl>\sigma_{0}}H_{+}^{2bl,l}(\Omega)
\end{equation}
such that $(A,B)u=(f,g_1,\ldots,g_m)$. To deduce the desired inclusion $u\in C^{\infty}_{+}(\overline{\Omega})$ from \eqref{9f5.21}, we use the extension operators $O$, $T_{\tau}$, and $T_{G}$ from the proof of Lemma~\ref{9lem5.1}. We can suppose that the mappings $T_{\tau}$ and $T_{G}$ are independent of $s>0$. The extension operators of this kind are constructed, e.g., in \cite{Rychkov99}. Then, according to \eqref{9f5.13} (with $s=2bl$ and $\gamma=1/(2b)$) and \eqref{9f5.21}, we can write
\begin{equation}\label{9f5.22}
w:=(T_{G}\otimes(OT_{\tau}))u\in
\bigcap_{l\in\mathbb{Z},\,2bl>\sigma_{0}}
H^{2bl,l}_{+}(\mathbb{R}^{n+1})\subset C^{\infty}(\mathbb{R}^{n+1}).
\end{equation}
Here, we note that $u\in\Upsilon^{2bl,l}(\Omega)$ by Lemma~\ref{9lem5.1} for $l$ indicated in \eqref{9f5.22}, the space $\Upsilon^{2bl,l}(\Omega)$ being defined in the proof of this lemma. We also remark that the letter inclusion in \eqref{9f5.22} holds true by the Sobolev embedding theorem (see also \cite[Part~II, Theorem~13]{Slobodeckii58}). Hence, $u=\nobreak w\!\upharpoonright\!\Omega\in C^{\infty}_{+}(\overline{\Omega})$. Thus, the mapping \eqref{9f2.4} is surjective. We have justified the bijection~\eqref{9f5.20}.

\section{Interpolation with a function parameter}\label{9sec6}

In this section we discuss the method of the interpolation with a function parameter between Hilbert spaces, which was introduced by C.~Foia\c{s} and J.-L.~Lions \cite[p.~278]{FoiasLions61}. This interpolation is a natural generalization of the classical interpolation method by S.~G.~Krein and J.-L.~Lions to the case where a sufficiently general function is used instead of a number as an interpolation parameter (see, e.g., monographs \cite[Chapter~IV, Section~1, Subsection~10]{KreinPetuninSemenov82} and \cite[Chapter~1, Sections 2 and 5]{LionsMagenes72i}). Interpolation with a function parameter will play a key role in the proof of Isomorphism Theorem. It is sufficient for our purposes to restrict the discussion to the case of separable complex Hilbert spaces. We follow the monograph \cite[Section~1.1]{MikhailetsMurach14}, which   systematically sets forth this interpolation (see also \cite[Section~2]{MikhailetsMurach08MFAT1}).

Let $X:=[X_{0},X_{1}]$ be an ordered pair of separable complex Hilbert spaces such that $X_{1}\subseteq X_{0}$ and this embedding is continuous and dense. This pair is said to be admissible. There exists a positive-definite self-adjoint operator $J$ on $X_{0}$ that has the domain $X_{1}$ and that satisfies the condition $\|Jv\|_{X_{0}}=\|v\|_{X_{1}}$ for every $v\in X_{1}$. This operator is uniquely determined by the pair $X$ and is called a generating operator for~$X$ (see, e.g., \cite[Chapter~IV, Theorem~1.12]{KreinPetuninSemenov82}). It defines an isometric isomorphism $J:X_{1}\leftrightarrow X_{0}$.

Let $\mathcal{B}$ denote the set of all Borel measurable functions $\psi:(0,\infty)\rightarrow(0,\infty)$ such that $\psi$ is bounded on each compact interval $[a,b]$, with $0<a<b<\infty$, and that $1/\psi$ is bounded on every semiaxis $[a,\infty)$, with $a>0$.

Given a function $\psi\in\mathcal{B}$, we consider the (generally, unbounded) operator $\psi(J)$, which is defined on $X_{0}$ as the Borel function $\psi$ of $J$. This operator is built with the help of Spectral Theorem applied to the self-adjoint operator $J$. Let $[X_{0},X_{1}]_{\psi}$ or, simply, $X_{\psi}$ denote the domain of the operator $\psi(J)$ endowed with the inner product
\begin{equation*}
(v_{1},v_{2})_{X_{\psi}}:=(\psi(J)v_{1},\psi(J)v_{2})_{X_{0}}.
\end{equation*}
The linear space $X_{\psi}$ is Hilbert and separable with respect to this inner product. The latter induces the norm $\|v\|_{X_{\psi}}:=\|\psi(J)v\|_{X_{0}}$.

A function $\psi\in\mathcal{B}$ is called an interpolation parameter if the following condition is fulfilled for all admissible pairs $X=[X_{0},X_{1}]$ and $Y=[Y_{0},Y_{1}]$ of Hilbert spaces and for an arbitrary linear mapping $T$ given on $X_{0}$: if the restriction of $T$ to $X_{j}$ is a bounded operator $T:X_{j}\rightarrow Y_{j}$ for each $j\in\{0,1\}$, then the restriction of $T$ to
$X_{\psi}$ is also a bounded operator $T:X_{\psi}\rightarrow Y_{\psi}$.

If $\psi$ is an interpolation parameter, then we say that the Hilbert space $X_{\psi}$ is obtained by the interpolation with the function parameter $\psi$ of the pair $X=\nobreak[X_{0},X_{1}]$ (or, otherwise speaking, between the spaces $X_{0}$ and $X_{1}$). In this case, we have the dense and continuous embeddings $X_{1}\hookrightarrow
X_{\psi}\hookrightarrow X_{0}$.

It is known that a function $\psi\in\mathcal{B}$ is an interpolation parameter if and only if $\psi$ is pseudoconcave in a neighbourhood of infinity, i.e. there exists a concave positive function $\psi_{1}(r)$ of $r\gg1$ such that both the functions $\psi/\psi_{1}$ and $\psi_{1}/\psi$ are bounded in some neighbourhood of infinity. This criterion follows from J.~Peetre's \cite{Peetre66, Peetre68} description of all interpolation functions for the weighted $L_{p}(\mathbb{R}^{n})$-type spaces (this result of J.~Peetre is also set forth in the monograph \cite[Theorem 5.4.4]{BerghLefstrem76}). The proof of the criterion is given in \cite[Section 1.1.9]{MikhailetsMurach14}.

We will use the next consequence of this criterion \cite[Theorem
1.11]{MikhailetsMurach14}.

\begin{proposition}\label{9prop6.1}
Suppose that a function $\psi\in\mathcal{B}$ varies regularly of index $\theta$ at infinity, with $0<\theta<1$, i.e.
$$
\lim_{r\rightarrow\infty}\;\frac{\psi(\lambda r)}{\psi(r)}=
\lambda^{\theta}\quad\mbox{for every}\quad\lambda>0.
$$
Then $\psi$ is an interpolation parameter.
\end{proposition}

The notion of a regularly varying function belongs to J.~Karamata~\cite{Karamata30a}. It is evident that a function $\psi:(r_{0},\infty)\rightarrow(0,\infty)$, with $r_{0}\in\mathbb{R}$, varies regularly of index $\theta\in\mathbb{R}$ at infinity if and only if $\psi(r)\equiv r^{\theta}\psi_{0}(r)$ for a certain function $\psi_{0}:\nobreak(r_{0},\infty)\rightarrow(0,\infty)$ that varies slowly at infinity, both functions being assumed to be Borel measurable.

Note that, in the case of power functions, Proposition~\ref{9prop6.1} leads us to the above-mentioned classical result by J.-L.~Lions and S.~G.~Krein. Namely, they proved that the function $\psi(r)\equiv\nobreak r^{\theta}$ is an interpolation parameter whenever $0<\theta<1$. In this case, the exponent $\theta$ is regarded as a number parameter of the interpolation.

We end this section with two properties of the interpolation, which
will be used in our proofs. The first of them enables us to reduce the interpolation of subspaces or factor spaces to the interpolation of initial spaces (see \cite[Sec. 1.1.6]{MikhailetsMurach14} or \cite[Sec. 1.17]{Triebel95}). Note that subspaces (of Hilbert spaces) are assumed to be closed and that we generally consider nonorthogonal projectors onto subspaces.

\begin{proposition}\label{9prop6.2}
Let $X=[X_{0},X_{1}]$ be an admissible pair of Hilbert spaces, and let $Y_{0}$ be
a subspace of $X_{0}$. Then $Y_{1}:=X_{1}\cap Y_{0}$ is a subspace of $X_{1}$.
Suppose that there exists a linear mapping $P:X_{0}\rightarrow X_{0}$ such that $P$
is a projector of the space $X_{j}$ onto its subspace $Y_{j}$ for every
$j\in\{0,\,1\}$. Then the pairs $[Y_{0},Y_{1}]$ and $[X_{0}/Y_{0},X_{1}/Y_{1}]$
are admissible, and
\begin{gather}\label{9f6.1}
[Y_{0},Y_{1}]_{\psi}=X_{\psi}\cap Y_{0},\\
[X_{0}/Y_{0},X_{1}/Y_{1}]_{\psi}=X_{\psi}/(X_{\psi}\cap Y_{0}) \label{9f6.2}
\end{gather}
with equivalence of norms for an arbitrary interpolation parameter~$\psi\in\mathcal{B}$. Here, $X_{\psi}\cap Y_{0}$ is a subspace of $X_{\psi}$.
\end{proposition}

The second property reduces the interpolation of ortogonal sums of Hilbert spaces to the interpolation of their summands.

\begin{proposition}\label{9prop6.3}
Let $[X_{0}^{(j)},X_{1}^{(j)}]$, with $j=1,\ldots,q$, be a finite collection of admissible pairs of Hilbert spaces. Then
$$
\biggl[\,\bigoplus_{j=1}^{q}X_{0}^{(j)},\,
\bigoplus_{j=1}^{q}X_{1}^{(j)}\biggr]_{\psi}=\,
\bigoplus_{j=1}^{q}\bigl[X_{0}^{(j)},\,X_{1}^{(j)}\bigr]_{\psi}
$$
with equality of norms for every function $\psi\in\mathcal{B}$.
\end{proposition}

\section{The interpolation between anisotropic Sobolev spaces}\label{9sec7}

The purpose of this section is to prove that the H\"ormander spaces appearing in Theorem~\ref{9th4.1} can be obtained by means of the interpolation with a function parameter between their Sobolev analogs.

We assume that
\begin{equation}\label{9f7.1}
s,s_{0},s_{1},\gamma\in\mathbb{R},\quad s_{0}<s<s_{1},\quad\gamma>0,
\quad\mbox{and}\quad\varphi\in\mathcal{M}.
\end{equation}
Consider the function
\begin{equation}\label{9f7.2}
\psi(r):=
\begin{cases}
\;r^{(s-s_{0})/(s_{1}-s_{0})}\,\varphi(r^{1/(s_{1}-s_{0})})&\text{for}\quad r\geq1, \\
\;\varphi(1) & \text{for}\quad0<r<1.
\end{cases}
\end{equation}
By Proposition~\ref{9prop6.1}, this function is an interpolation parameter  because $\psi$ varies regularly of index $\theta:=(s-s_{0})/(s_{1}-s_{0})$ at infinity, with $0<\theta<1$. We will interpolate pairs of Sobolev spaces with the function parameter $\psi$.

Beforehand, we will derive the necessary interpolation formulas for the basic H\"ormander spaces $H^{s,s\gamma;\varphi}(\mathbb{R}^{k+1})$ and $H^{s,s\gamma;\varphi}_{+}(\mathbb{R}^{k+1})$, with the integer $k\geq1$. All the results of this section are formulated as lemmas.

\begin{lemma}\label{9lem7.1}
On the assumption \eqref{9f7.1} we have
\begin{equation}\label{9f7.3}
H^{s,s\gamma;\varphi}(\mathbb{R}^{k+1})=
\bigl[H^{s_{0},s_{0}\gamma}(\mathbb{R}^{k+1}),
H^{s_{1},s_{1}\gamma}(\mathbb{R}^{k+1})\bigr]_{\psi}
\end{equation}
with equality of norms.
\end{lemma}

The proof of this lemma is analogous to the proof of \cite[Lemma~5.1]{LosMurach13MFAT2}, where the $k=1$ case was examined. Nevertheless, we give this proof for the reader's convenience.

\begin{proof}
According to \eqref{9f3.1}, the pair
\begin{equation*}
X:=\bigl[H^{s_{0},s_{0}\gamma}(\mathbb{R}^{k+1}),
H^{s_{1},s_{1}\gamma}(\mathbb{R}^{k+1})\bigr]
\end{equation*}
of anisotropic Sobolev spaces is admissible. It follows immediately from their definition that the generating operator for $X$ is given by the formula
$$
J:\,w\mapsto\mathcal{F}^{-1}[r_{\gamma}^{s_{1}-s_{0}}\,\mathcal{F}w\,],
\quad\mbox{with}\quad w\in H^{s_{1},s_{1}\gamma}(\mathbb{R}^{k+1}).
$$
Here, $\mathcal{F}$ [$\mathcal{F}^{-1}$ respectively] denotes the operator of the direct [inverse] Fourier transform in all variables of tempered distributions given in $\mathbb{R}^{k+1}$.

The generating operator $J$ is reduced to the operator of multiplication by $r_{\gamma}^{s_{1}-s_{0}}$ by means of the Fourier transform which sets an isometric isomorphism
$$
\mathcal{F}:\,H^{s_{0},s_{0}\gamma}(\mathbb{R}^{k+1})\leftrightarrow
L_{2}\bigl(\mathbb{R}^{k+1},r_{\gamma}^{2s_{0}}(\xi,\eta)d\xi d\eta\bigr).
$$
Here, of course, the second Hilbert space consists of all functions (of $\xi\in\mathbb{R}^{k}$ and $\eta\in\mathbb{R}$) that are square integrable over $\mathbb{R}^{k+1}$ with respect to the Radon measure  $r_{\gamma}^{2s_{0}}(\xi,\eta)d\xi d\eta$. Hence, $\mathcal{F}$ reduces $\psi(J)$ to the operator of multiplication by the function
$$
\psi(r_{\gamma}^{s_{1}-s_{0}}(\xi,\eta))\equiv
r_{\gamma}^{s-s_{0}}(\xi,\eta)\,\varphi(r_{\gamma}(\xi,\eta)),
$$
the identity being due to \eqref{9f7.2}.

Therefore, given $w\in C^{\infty}_{0}(\mathbb{R}^{k+1})$, we can write
\begin{align*}
\|w\|_{X_{\psi}}^{2}&=
\|\psi(J)w\|_{H^{s_{0},s_{0}\gamma}(\mathbb{R}^{k+1})}^{2}\\
&=\int\limits_{\mathbb{R}^{k}}\,\int\limits_{\mathbb{R}}
r_{\gamma}^{2s_{0}}(\xi,\eta)\,
|\psi(r_{\gamma}^{s_{1}-s_{0}}(\xi,\eta))\,(\mathcal{F}w)(\xi,\eta)|^{2}\,
d\xi d\eta\\&=\|w\|_{H^{s,s\gamma;\varphi}(\mathbb{R}^{k+1})}^{2}.
\end{align*}
This implies the equality of spaces \eqref{9f7.3} as the set $C^{\infty}_{0}(\mathbb{R}^{k+1})$
is dense in both of them. Here, we remark that this set is dense in the second space denoted by $X_{\psi}$ because $C^{\infty}_{0}(\mathbb{R}^{k+1})$ is
dense in the space $H^{s_{1},s_{1}\gamma}(\mathbb{R}^{k+1})$ embedded continuously and densely in $X_{\psi}$.
\end{proof}

\begin{lemma}\label{9lem7.2}
Assume in addition to \eqref{9f7.1} that $s_{0}\geq0$. Then
\begin{equation}\label{9f7.4}
H^{s,s\gamma;\varphi}_{+}(\mathbb{R}^{k+1})=
\bigl[H^{s_{0},s_{0}\gamma}_{+}(\mathbb{R}^{k+1}),
H^{s_{1},s_{1}\gamma}_{+}(\mathbb{R}^{k+1})\bigr]_{\psi}
\end{equation}
with equivalence of norms.
\end{lemma}

\begin{proof}
We will deduce this lemma from Lemma~\ref{9lem7.1} with the help of Proposition~\ref{9prop6.2}. To this end, we need to present a linear mapping $P$ on $L_{2}(\mathbb{R}^{k+1})$ that the restriction of $P$ to each space $H^{s_{j},s_{j}\gamma}(\mathbb{R}^{k+1})$, with $j\in\{0,1\}$, is a projector of $H^{s_{j},s_{j}\gamma}(\mathbb{R}^{k+1})$ onto $H^{s_{j},s_{j}\gamma}_{+}(\mathbb{R}^{k+1})$.

Let us consider a bounded linear operator
\begin{equation}\label{9f7.5}
T:L_{2}((-\infty,0))\rightarrow L_{2}(\mathbb{R})
\end{equation}
such that $Th=h$ on $(-\infty,0)$ for every function $h\in L_{2}((-\infty,0))$ and that the restriction of the mapping $T$ to the Sobolev space $H^{s_{j}\gamma}((-\infty,0))$ is a bounded operator
\begin{equation}\label{9f7.6}
T:H^{s_{j}\gamma}((-\infty,0))\rightarrow H^{s_{j}\gamma}(\mathbb{R})\quad
\mbox{for each}\quad j\in\{0,1\}.
\end{equation}
The operator $T$ exists \cite[Lemma 2.9.3]{Triebel95}. Using the tensor product of bounded operators on Hilbert spaces, we obtain the bounded linear operator
\begin{equation}\label{9f7.7}
I\otimes T:L_{2}(\mathbb{R}^{k}\times(-\infty,0))\to
L_{2}(\mathbb{R}^{k+1})
\end{equation}
such that $(I\otimes T)v=v$ on
$\mathbb{R}^{k}\times(-\infty,0)$ for every function $v\in L_{2}(\mathbb{R}^{k}\times(-\infty,0))$. Here, $I$ is the identity operator on $L_{2}(\mathbb{R}^{k})$.

Given $j\in\{0,1\}$, we can write
\begin{equation}\label{9f7.8}
\begin{aligned}
&H^{s_{j},s_{j}\gamma}(\mathbb{R}^{k}\times(-\infty,0))\\
=&H^{s_{j}}(\mathbb{R}^{k})\otimes L_{2}((-\infty,0))\cap
L_{2}(\mathbb{R}^{k})\otimes H^{s_{j}\gamma}((-\infty,0))
\end{aligned}
\end{equation}
and
\begin{equation}\label{9f7.9}
H^{s_{j},s_{j}\gamma}(\mathbb{R}^{k+1})=
H^{s_{j}}(\mathbb{R}^{k})\otimes L_{2}(\mathbb{R})\cap
L_{2}(\mathbb{R}^{k})\otimes H^{s_{j}\gamma}(\mathbb{R}).
\end{equation}
These equalities of spaces hold true up to equivalence of norms. Being based on \eqref{9f7.5}, \eqref{9f7.6}, \eqref{9f7.8}, and \eqref{9f7.9},  we conclude that the restriction of \eqref{9f7.7} to the space $H^{s_{j},s_{j}\gamma}(\mathbb{R}^{k}\times(-\infty,0))$ is a bounded operator
\begin{equation}\label{9f7.10}
I\otimes T:H^{s_{j},s_{j}\gamma}(\mathbb{R}^{k}\times(-\infty,0))\to
H^{s_{j},s_{j}\gamma}(\mathbb{R}^{k+1}).
\end{equation}

Consider the linear mapping
\begin{equation*}
P:w\mapsto w-(I\otimes T)\bigl(w\!\upharpoonright\!(\mathbb{R}^{k}\times(-\infty,0))\bigr),\quad
\mbox{with}\quad w\in L_{2}(\mathbb{R}^{k+1}).
\end{equation*}
It is easy to see that $\mathrm{supp}\,Pw\subseteq\mathbb{R}^{k}\times[0,\infty)$ and that the inclusion $\mathrm{supp}\,w\subseteq\mathbb{R}^{k}\times[0,\infty)$ implies the equality $Pw=w$ on $\mathbb{R}^{k+1}$. Using these properties of $P$ and the boundedness of the operator \eqref{9f7.10}, we conclude that the mapping $P$ is required.

Now, by virtue of Proposition~\ref{9prop6.2} (formula \eqref{9f6.1}) and Lemma~\ref{9lem7.1}, we can write
\begin{align*}
\bigl[H&^{s_{0},s_{0}\gamma}_{+}(\mathbb{R}^{k+1}),
H^{s_{1},s_{1}\gamma}_{+}(\mathbb{R}^{k+1})\bigr]_{\psi}\\
=&\bigl[H^{s_{0},s_{0}\gamma}(\mathbb{R}^{k+1}),
H^{s_{1},s_{1}\gamma}(\mathbb{R}^{k+1})\bigr]_{\psi}\cap
H^{s_{0},s_{0}\gamma}_{+}(\mathbb{R}^{k+1})\\
=&H^{s,s\gamma;\varphi}(\mathbb{R}^{k+1})\cap
H^{s_{0},s_{0}\gamma}_{+}(\mathbb{R}^{k+1})\\
=&H^{s,s\gamma;\varphi}_{+}(\mathbb{R}^{k+1}).
\end{align*}
These equalities of spaces hold true up to equivalence of norms. (Note that the first pair is admissible by Proposition~\ref{9prop6.2}.) Thus, we have proved \eqref{9f7.4}.
\end{proof}

\begin{lemma}\label{9lem7.3}
Assume in addition to \eqref{9f7.1} that $s_{0}\geq0$ and
\begin{equation}\label{9f7.11}
s_{j}\gamma-1/2\notin\mathbb{Z}\quad\mbox{for each}\quad j\in\{0,1\}.
\end{equation}
Then
\begin{equation}\label{9f7.12}
H^{s,s\gamma;\varphi}_{+}(\Omega)=
\bigl[H^{s_{0},s_{0}\gamma}_{+}(\Omega),
H^{s_{1},s_{1}\gamma}_{+}(\Omega)\bigr]_{\psi}
\end{equation}
and
\begin{equation}\label{9f7.13}
H^{s,s\gamma;\varphi}_{+}(S)=
\bigl[H^{s_{0},s_{0}\gamma}_{+}(S),
H^{s_{1},s_{1}\gamma}_{+}(S)\bigr]_{\psi}
\end{equation}
with equivalence of norms.
\end{lemma}

\begin{proof}
We will first prove \eqref{9f7.12}. Recall that, by definition,
\begin{equation}\label{9f7.14}
H^{s,s\gamma;\varphi}_{+}(\Omega)=
H^{s,s\gamma;\varphi}_{+}(\mathbb{R}^{n+1})/
H^{s,s\gamma;\varphi}_{+}(\mathbb{R}^{n+1},\Omega)
\end{equation}
and
\begin{equation}\label{9f7.15}
H^{s_{j},s_{j}\gamma}_{+}(\Omega)=
H^{s_{j},s_{j}\gamma}_{+}(\mathbb{R}^{n+1})/
H^{s_{j},s_{j}\gamma}_{+}(\mathbb{R}^{n+1},\Omega)
\end{equation}
for each $j\in\{0,1\}$. Here, the denominators are defined by \eqref{9f3.5}. We will deduce formula \eqref{9f7.12} from Lemma~\ref{9lem7.2} with the help of Proposition~\ref{9prop6.2}, the interpolation of factor spaces being used. For this purpose, we need to present a linear mapping $P$ on $H^{s_{0},s_{0}\gamma}_{+}(\mathbb{R}^{n+1})$ that the restriction of $P$ to each $H^{s_{j},s_{j}\gamma}_{+}(\mathbb{R}^{n+1})$, with $j\in\{0,1\}$, is a projector of the space $H^{s_{j},s_{j}\gamma}_{+}(\mathbb{R}^{n+1})$ onto its subspace $H^{s_{j},s_{j}\gamma}_{+}(\mathbb{R}^{n+1},\Omega)$.

Let us make use of the reasoning and notation given in the proof of Lemma~\ref{9lem5.1}. The justification of \eqref{9f5.13} presented in this proof shows also that we have the bounded operator
\begin{equation}\label{9f7.16}
T_{+}:=T_{G}\otimes(OT_{\tau}):\Upsilon^{s_{j},s_{j}\gamma}(\Omega)\to
H^{s_{j},s_{j}\gamma}_{+}(\mathbb{R}^{n+1})
\end{equation}
for each $j\in\{0,1\}$. Here, the operators $T_{G}$ and $T_{\tau}$ respectively are restrictions of the mappings \eqref{9f5.7} and \eqref{9f5.9}, which do not depend on $j$ (see \cite[Theorems 4.2.2 and 4.2.3]{Triebel95}). Therefore the operator \eqref{9f7.16} with $j=1$ is a restriction of its counterpart with $j=0$. Besides, as we have mentioned just after \eqref{9f5.13}, the equality $T_{+}u=u$ holds on $\Omega$ for every $u\in\Upsilon^{s_{j},s_{j}\gamma}(\Omega)$. Note also that $\Upsilon^{s_{j},s_{j}\gamma}(\Omega)=
H^{s_{j},s_{j}\gamma}_{+}(\Omega)$ due to Lemma~\ref{9lem5.1} and the condition \eqref{9f7.11}.

Let us consider the linear mapping
\begin{equation*}
P:w\mapsto w-T_{+}(w\!\upharpoonright\!\Omega),\quad\mbox{with}\quad
w\in H^{s_{0},s_{0}\gamma}_{+}(\mathbb{R}^{n+1}).
\end{equation*}
Note that $Pw=0$ on $\Omega$ and that the condition $w=0$ on $\Omega$ implies the equality $Pw=w$ on $\mathbb{R}^{n+1}$. Taking into account these properties of $P$ and the boundedness of the operator~\eqref{9f7.16}, we conclude that the mapping $P$ is required.

Now, using successively \eqref{9f7.15}, Proposition~\ref{9prop6.2} (formula \eqref{9f6.2}), Lemma~\ref{9lem7.2}, and \eqref{9f7.14}, we write the following:
\begin{align*}
\bigl[H&^{s_{0},s_{0}\gamma}_{+}(\Omega),
H^{s_{1},s_{1}\gamma}_{+}(\Omega)\bigr]_{\psi}\\
=&\bigl[H^{s_{0},s_{0}\gamma}_{+}(\mathbb{R}^{n+1})/
H^{s_{0},s_{0}\gamma}_{+}(\mathbb{R}^{n+1},\Omega),
H^{s_{1},s_{1}\gamma}_{+}(\mathbb{R}^{n+1})/
H^{s_{1},s_{1}\gamma}_{+}(\mathbb{R}^{n+1},\Omega)\bigr]_{\psi}\\
=&X_{\psi}/(X_{\psi}\cap H^{s_{0},s_{0}\gamma}_{+}(\mathbb{R}^{n+1},\Omega))\\
=&H^{s,s\gamma;\varphi}_{+}(\mathbb{R}^{n+1})/
(H^{s,s\gamma;\varphi}_{+}(\mathbb{R}^{n+1})\cap
H^{s_{0},s_{0}\gamma}_{+}(\mathbb{R}^{n+1},\Omega))\\
=&H^{s,s\gamma;\varphi}_{+}(\mathbb{R}^{n+1})/
H^{s,s\gamma;\varphi}_{+}(\mathbb{R}^{n+1},\Omega)\\
=&H^{s,s\gamma,\varphi}_{+}(\Omega).
\end{align*}
Here,
\begin{equation*}
X_{\psi}:=\bigl[H^{s_{0},s_{0}\gamma}_{+}(\mathbb{R}^{n+1}),
H^{s_{1},s_{1}\gamma}_{+}(\mathbb{R}^{n+1})\bigr]_{\psi}=
H^{s,s\gamma;\varphi}_{+}(\mathbb{R}^{n+1}).
\end{equation*}
These equalities of spaces hold true up to equivalence of norms. (Remark also that the first pair is admissible by Proposition~\ref{9prop6.2}.) Thus, we have proved \eqref{9f7.12}.

Let us prove the interpolation formula \eqref{9f7.13}. The pair of spaces on the right of \eqref{9f7.13} is admissible due to Lemma~\ref{9lem3.1}. (Its proof given at the end of this section does not use this lemma in the case where $\varphi(r)\equiv1$ and $s\gamma-1/2\notin\mathbb{Z}$.) We will deduce formula \eqref{9f7.13} from its analog
\begin{equation}\label{9f7.17}
H^{s,s\gamma;\varphi}_{+}(\Pi)=
\bigl[H^{s_{0},s_{0}\gamma}_{+}(\Pi),
H^{s_{1},s_{1}\gamma}_{+}(\Pi)\bigr]_{\psi}
\end{equation}
for $\Pi:=\mathbb{R}^{n-1}\times(0,\tau)$. The proof of \eqref{9f7.17} is the same as that of \eqref{9f7.12}, but with $\Pi$, $\mathbb{R}^{n}$, and the identity operator instead of $\Omega$, $\mathbb{R}^{n+1}$, and $T_{G}$ respectively.

Using the definition of anisotropic H\"ormander spaces over $S$ given in \eqref{9f3.7} and \eqref{9f3.8}, we will deduce \eqref{9f7.13} from \eqref{9f7.17} with the help of certain operators of flattening and sewing of the manifold $S$. We define the flattening operator by the formula
\begin{equation*}
L:v\mapsto((\chi^*_{1}v)\circ\theta_{1}^*,\ldots,
(\chi^*_{\lambda}v)\circ\theta_{\lambda}^*),
\quad\mbox{with}\quad v\in L_2(S).
\end{equation*}
Its restrictions to the spaces $H^{s,s\gamma;\varphi}_{+}(S)$ and $H^{s_j,s_j\gamma}_{+}(S)$ are isometric linear operators
\begin{equation}\label{9f7.18}
L:H^{s,s\gamma;\varphi}_{+}(S)\rightarrow
\bigl(H^{s,s\gamma;\varphi}_{+}(\Pi)\bigr)^{\lambda}
\end{equation}
and
\begin{equation}\label{9f7.19}
L:H^{s_j,s_j\gamma}_{+}(S)\rightarrow
\bigl(H^{s_j,s_j\gamma}_{+}(\Pi)\bigr)^{\lambda}
\quad\mbox{for each}\quad j\in\{0,1\}.
\end{equation}
This follows immediately from the definition of these spaces. Interpolating with the function parameter $\psi$ between the spaces in \eqref{9f7.19}, we obtain a bounded operator
\begin{equation}\label{9f7.20}
L:\bigl[H^{s_0,s_0\gamma}_{+}(S),H^{s_1,s_1\gamma}_{+}(S)\bigr]_{\psi}\to
\bigl[\bigl(H^{s_0,s_0\gamma}_{+}(\Pi)\bigr)^{\lambda},
\bigl(H^{s_1,s_1\gamma}_{+}(\Pi)\bigr)^{\lambda}\bigr]_{\psi}.
\end{equation}
The latter interpolation space is equal to $(H^{s,s\gamma;\varphi}_{+}(\Pi))^{\lambda}$ up to equivalence of norms by virtue of Proposition~\ref{9prop6.3} and formula~\eqref{9f7.17}. Hence, \eqref{9f7.20} is a bounded operator
\begin{equation}\label{9f7.21}
L:\bigl[H^{s_0,s_0\gamma}_{+}(S),H^{s_1,s_1\gamma}_{+}(S)\bigr]_{\psi}\to
\bigl(H^{s,s\gamma;\varphi}_{+}(\Pi)\bigr)^{\lambda}.
\end{equation}

We define the sewing operator by the formula
\begin{equation*}
K:(h_{1},\ldots,h_{\lambda})\mapsto\sum_{k=1}^{\lambda}\,
O_{k}((\eta_{k}^{*}h_{k})\circ\theta_{k}^{*-1}),
\;\;\mbox{with}\;\;h_{1},\ldots,h_{\lambda}\in L_2(\Pi).
\end{equation*}
Here, each function $\eta_{k}\in C_{0}^{\infty}(\mathbb{R}^{n-1})$ is chosen so that $\eta_{k}=1$ on the set $\theta^{-1}_{k}(\mathrm{supp}\,\chi_{k})$; next $\eta^{*}_{k}(x,t):=\eta_{k}(x)$ for all $x\in\mathbb{R}^{n-1}$ and $t\in (0,\tau)$. Besides, $O_{k}$ denotes the operator of the extension (of functions) by zero from $\Gamma_k\times(0,\tau)$ to $S$; thus, for every $y\in\Gamma$ and $t\in (0,\tau)$, we have
\begin{equation*}
\bigl(O_{k}((\eta_{k}^{*}h_{k})\circ\theta_{k}^{*-1})\bigr)(y,t)=
\left\{
  \begin{array}{ll}
    \eta_{k}(x)h_{k}(x,t)&\mbox{if}\;\, y=\theta_{k}(x)\in\Gamma_{k}\;\,\mbox{with}\;\,x\in\mathbb{R}^{n-1};\\
    0&\hbox{otherwise.}
  \end{array}
\right.
\end{equation*}

The mapping $K$ is left inverse to $L$. Indeed,
\begin{align*}
KLv&=\sum_{k=1}^{\lambda}\,O_{k}\bigl(\bigl(\eta^*_{k}\,
((\chi_{k}^*v)\circ\theta_{k}^*)\bigr)\circ\theta_{k}^{*-1}\bigr)\\
&=\sum_{k=1}^{\lambda}\,O_{k}
\bigl((\chi_{k}^*v)\circ\theta_{k}^*\circ\theta_{k}^{*-1}\bigr)=
\sum_{k=1}^{\lambda}\,\chi_{k}^*v=v,
\end{align*}
that is
\begin{equation}\label{9f7.22}
KLv=v\quad\mbox{for every}\quad v\in L_2(S).
\end{equation}

Let us show that the restriction of the linear mapping $K$ to the space $(H^{s,s\gamma;\varphi}_{+}(\Pi))^{\lambda}$ is a bounded operator
\begin{equation}\label{9f7.23}
K:\bigl(H^{s,s\gamma;\varphi}_{+}(\Pi)\bigr)^{\lambda}
\rightarrow H^{s,s\gamma;\varphi}_{+}(S).
\end{equation}
Given a vector-valued function
\begin{equation*}
h:=(h_{1},\ldots,h_{\lambda})\in
\bigl(H^{s,s\gamma;\varphi}_{+}(\Pi)\bigr)^{\lambda},
\end{equation*}
we can write
\begin{align}\label{9f7.24}\notag
\bigl\|Kh\bigr\|^{2}_{H^{s,s\gamma;\varphi}_{+}(S)}
&=\sum_{l=1}^{\lambda}\;\bigl\|(\chi_{l}^*\,Kh)
\circ\theta_{l}^*\bigr\|_{H^{s,s\gamma;\varphi}_{+}(\Pi)}^{2}\\ \notag
&=\sum_{l=1}^{\lambda}\,\Bigl\|\Bigl(\chi_{l}^*\,
\sum_{k=1}^{\lambda}\,
O_{k}\left((\eta_{k}^{*}h_{k})\circ\theta_{k}^{*-1}\right)
\Bigr)
\circ\theta_{l}^*\,\Bigr\|_{H^{s,s\gamma;\varphi}_{+}(\Pi)}^{2}\\
&=\sum_{l=1}^{\lambda}\;\Bigl\|\,
\sum_{k=1}^{\lambda}Q_{k,l}\,h_{k}\,
\Bigr\|_{H^{s,s\gamma;\varphi}_{+}(\Pi)}^{2}.
\end{align}
Here, we define
\begin{equation}\label{9f7.25}
(Q_{k,l}\,w)(x,t):=\eta_{k,l}(\beta_{k,l}(x))\,w(\beta_{k,l}(x),t)
\end{equation}
for all $w\in L_{2}(\Pi)$, $x\in\mathbb{R}^{n-1}$, and $t\in(0,\tau)$, where
\begin{equation*}
\eta_{k,l}:=(\chi_{l}\circ\theta_{k})\eta_{k}\in C_{0}^{\infty}(\mathbb{R}^{n-1})
\end{equation*}
and, moreover, $\beta_{k,l}:\mathbb{R}^{n-1}\leftrightarrow\mathbb{R}^{n-1}$ is an infinitely smooth diffeomorphism such that $\beta_{k,l}=\nobreak\theta_{k}^{-1}\circ\theta_{l}$ in a neighbourhood of $\mathrm{supp}\,\eta_{k,l}$. As is known \cite[Theorem~B.1.8]{Hormander85}, the operator $\omega\mapsto(\eta_{k,l}\,\omega)\circ\beta_{k,l}$ is bounded on every Sobolev space $H^{\sigma}(\mathbb{R}^{n-1})$ with $\sigma\in\mathbb{R}$. Therefore the operator $w\mapsto Q_{k,l}\,w$ defined by formula \eqref{9f7.25} for all $w\in L_{2}(\mathbb{R}^{n})$, $x\in\mathbb{R}^{n-1}$, and $t\in\mathbb{R}$ is bounded on each space
\begin{equation*}
H^{s_{j},s_{j}\gamma}(\mathbb{R}^{n})=H^{s_{j}}(\mathbb{R}^{n-1})\otimes L_{2}(\mathbb{R})\cap
L_{2}(\mathbb{R}^{n-1})\otimes H^{s_{j}\gamma}(\mathbb{R})
\end{equation*}
with $j\in\{0,1\}$. Hence, the restriction of the mapping $w\mapsto Q_{k,l}w$, with $w\in L_{2}(\Pi)$, to each space $H^{s_{j},s_{j}\gamma}_{+}(\Pi)$ is a bounded operator on this space.
Then, owing to the interpolation formula \eqref{9f7.17}, the restriction of this mapping to the space $H^{s,s\gamma;\varphi}_{+}(\Pi)$ is a bounded operator on this space. Combining the latter conclusion with \eqref{9f7.24}, we can write
\begin{equation*}
\bigl\|Kh\bigr\|^{2}_{H^{s,s\gamma;\varphi}_{+}(S)}=
\sum_{l=1}^{\lambda}\;\Bigl\|\,
\sum_{k=1}^{\lambda}Q_{k,l}\,h_{k}\,
\Bigr\|_{H^{s,s\gamma;\varphi}_{+}(\Pi)}^{2}\leq
c\sum_{k=1}^{\lambda}\|h_{k}\|_{H^{s,s\gamma;\varphi}_{+}(\Pi)}^{2}
\end{equation*}
for some number $c>0$ that does not depend on $h$. Thus, we have proved the boundedness of the operator \eqref{9f7.23}.

The same reasoning also proves that the restriction of the mapping $K$ to the space $(H^{s_{j},s_{j}\gamma}_{+}(\Pi))^{\lambda}$ is a bounded operator
\begin{equation}\label{9f7.26}
K:\bigl(H^{s_{j},s_{j}\gamma}_{+}(\Pi)\bigr)^{\lambda}
\rightarrow H^{s_{j},s_{j}\gamma}_{+}(S)
\quad\mbox{for each}\quad j\in\{0,1\}.
\end{equation}
Interpolating between the spaces in \eqref{9f7.26} with the function parameter $\psi$ and using Proposition~\ref{9prop6.3} and formula~\eqref{9f7.17}, we obtain a bounded operator
\begin{equation}\label{9f7.27}
K:\bigl(H^{s,s\gamma;\varphi}_{+}(\Pi)\bigr)^{\lambda}\to
\bigl[H^{s_0,s_0\gamma}_{+}(S),H^{s_1,s_1\gamma}_{+}(S)\bigr]_{\psi}.
\end{equation}

Now it follows directly from \eqref{9f7.18}, \eqref{9f7.27}, and \eqref{9f7.22} that the identity operator $KL$ realizes a continuous embedding of the space $H^{s,s\gamma,\varphi}_{+}(S)$ in the interpolation space
\begin{equation*}
\bigl[H^{s_0,s_0\gamma}_{+}(S),H^{s_1,s_1\gamma}_{+}(S)\bigr]_{\psi}.
\end{equation*}
Moreover, it follows immediately from \eqref{9f7.21} and \eqref{9f7.23} that the identity operator $KL$ establishes the inverse continuous embedding. Thus, we have proved that the equality \eqref{9f7.13} is true up to equivalence of norms.
\end{proof}

\begin{remark}\label{9rem7.4}\rm
Lemma \ref{9lem7.3} remains valid without assumption \eqref{9f7.11}. Indeed, if  $s_{j}\gamma-1/2\notin\mathbb{Z}$ for some $j\in\{0,1\}$, then the interpolation formulas \eqref{9f7.12} and \eqref{9f7.13} can be deduced from this lemma with the help of the reiteration property of the interpolation \cite[Theorem~1.3]{MikhailetsMurach14}.
\end{remark}

At the end of this section, we will prove Lemma \ref{9lem3.1}.

\begin{proof}[Proof of Lemma $\ref{9lem3.1}$]
We first examine the case where $\varphi(r)\equiv1$ and $s\gamma-1/2\notin\mathbb{Z}$. Lemma~\ref{9lem5.1} implies that $H^{s,s\gamma}_{+}(S)$ is equal up to equivalence of norms to a certain subspace of $H^{s,s\gamma}(S)$. Therefore, assertion~(i) is a known property of the anisotropic Sobolev space $H^{s,s\gamma}(S)$; see \cite[Chapter~I, \S~5]{Slobodeckii58}.

Let us prove assertion~(ii) with the help of the flattening operator $L$ and sewing operator $K$ used in the proof of Lemma~\ref{9lem7.3}. As has been stated in Section~\ref{9sec3}, the set
\begin{equation*}
\Upsilon^{\infty}_{0}(\overline{\Pi}):=\bigl\{w\!\upharpoonright\!\overline{\Pi}:
w\in C^{\infty}_{0}(\mathbb{R}^{n-1}\times(0,\infty))\bigr\}
\end{equation*}
is dense in $H^{s,s\gamma}_{+}(\Pi)$. Given a function $v\in H^{s,s\gamma}_{+}(S)$, we approximate the vector $Lv$ by the sequence of vectors $h^{(j)}\in(\Upsilon^{\infty}_{0}(\overline{\Pi}))^{\lambda}$ in the norm of the space $(H^{s,s\gamma}_{+}(\Pi)\bigr)^{\lambda}$. Then the sequence of functions $Kh^{(j)}\in C^{\infty}_{+}(\overline{S})$ approximates the function $KLv=v$ in the norm of the space $H^{s,s\gamma}_{+}(S)$. Assertion (ii) is proved in the case examined.

In the general situation, Lemma~$\ref{9lem3.1}$ follows from this case with the help of Lemma~\ref{9lem7.3}. Indeed, the space $H^{s,s\gamma;\varphi}_{+}(S)$ is complete and separable due to properties of the interpolation used in formula~\eqref{9f7.13}. Moreover, let $\mathcal{A}_{1}$ and $\mathcal{A}_{2}$ be two pairs each of which consists of an atlas on $\Gamma$ and a partition of unity used in the definitions \eqref{9f3.7} and \eqref{9f3.8}. Let $H^{s,s\gamma;\varphi}_{+}(S,\mathcal{A}_{k})$ and $H^{s_{j},s_{j}\gamma}_{+}(S,\mathcal{A}_{k})$, with $j\in\{0,1\}$, denote the spaces $H^{s,s\gamma;\varphi}_{+}(S)$ and
$H^{s_{j},s_{j}\gamma}_{+}(S)$ corresponding to the pair $\mathcal{A}_{k}$ with $k\in\{0,1\}$. As has been stated in the previous paragraph, the identity mapping $I$ is an isomorphism between the spaces $H^{s_{j},s_{j}\gamma}_{+}(S,\mathcal{A}_{0})$ and $H^{s_{j},s_{j}\gamma}_{+}(S,\mathcal{A}_{1})$ for each $j\in\{0,1\}$. Therefore, by virtue of the interpolation formula \eqref{9f7.13}, we have the isomorphism
\begin{equation*}
I:H^{s,s\gamma;\varphi}_{+}(S,\mathcal{A}_{0})\leftrightarrow
H^{s,s\gamma;\varphi}_{+}(S,\mathcal{A}_{1}).
\end{equation*}
This means that the space $H^{s,s\gamma;\varphi}_{+}(S)$ does not depend on the choice of an atlas on $\Gamma$ and a partition of unity. Finally, Assertion~(ii) in the general situation results from the density of $H^{s_{1},s_{1}\gamma}_{+}(S)$ in $H^{s,s\gamma;\varphi}_{+}(S)$. (However, this assertion can be proved in the same way as that in the previous paragraph.)
\end{proof}

\section{Proofs of the main results}\label{9sec8}

In this section, we will prove Theorems \ref{9th4.1}, \ref{9th4.3}, and
\ref{9th4.4}. We will also justify Remark~\ref{9rem4.5}.

\begin{proof}[Proof of Theorem $\ref{9th4.1}$.] Let $\sigma>\sigma_0$ and $\varphi\in\mathcal{M}$. We choose an integer $\sigma_1>\sigma$ so that $\sigma_1/(2b)\in\mathbb{Z}$. According to M.~S.~Agranovich and M.~I.~Vishik's result \cite[Theorem~12.1]{AgranovichVishik64}, the mapping \eqref{9f2.4} extends uniquely (by continuity) to isomorphisms
\begin{equation}\label{9f8.1}
(A,B):H^{\sigma_k,\sigma_k/(2b)}_{+}(\Omega)\leftrightarrow
\mathcal{H}^{\sigma_k-2m,(\sigma_k-2m)/(2b)}_{+}(\Omega,S)
\quad\mbox{with}\quad k\in\{0,1\}.
\end{equation}
(Here, recall, the second space is defined by formula \eqref{9f4.2} with $\sigma:=\sigma_k$ and $\varphi\equiv1$.)

Let us define the interpolation parameter $\psi$ by formula \eqref{9f7.2} in which $s:=\sigma$, $s_{0}:=\sigma_{0}$, and $s_{1}:=\sigma_{1}$. Interpolating with the function parameter $\psi$ between the spaces in \eqref{9f8.1}, we obtain an isomorphism
\begin{equation}\label{9f8.2}
\begin{aligned}
(A,B)&:\bigl[H^{\sigma_0,\sigma_0/(2b)}_{+}(\Omega),
H^{\sigma_1,\sigma_1/(2b)}_{+}(\Omega)\bigr]_{\psi}\\
&\leftrightarrow
\bigl[\mathcal{H}^{\sigma_0-2m,(\sigma_0-2m)/(2b)}_{+}(\Omega,S),
\mathcal{H}^{\sigma_1-2m,(\sigma_1-2m)/(2b)}_{+}(\Omega,S)\bigr]_{\psi}.
\end{aligned}
\end{equation}
It is a restriction of the operator \eqref{9f8.1} with $k=0$.

According to Lemma~\ref{9lem7.3} and Proposition~\ref{9prop6.3} we can write
\begin{equation*}
\bigl[H^{\sigma_0,\sigma_0/(2b)}_{+}(\Omega),
H^{\sigma_1,\sigma_1/(2b)}_{+}(\Omega)\bigr]_{\psi}=
H^{\sigma,\sigma/(2b);\varphi}_{+}(\Omega)
\end{equation*}
and
\begin{align*}
[&\mathcal{H}^{\sigma_0-2m,(\sigma_0-2m)/(2b)}_{+}(\Omega,S),
\mathcal{H}^{\sigma_1-2m,(\sigma_1-2m)/(2b)}_{+}(\Omega,S)]_{\psi}\\
&=\bigl[H^{\sigma_0-2m,(\sigma_0-2m)/(2b)}_{+}(\Omega),
H^{\sigma_1-2m,(\sigma_1-2m)/(2b)}_{+}(\Omega)\bigr]_{\psi}\\
&\quad\oplus\bigoplus_{j=1}^{m}
\bigl[H^{\sigma_0-m_j-1/2,(\sigma_0-m_j-1/2)/(2b)}_{+}(S),
H^{\sigma_1-m_j-1/2,(\sigma_1-m_j-1/2)/(2b)}_{+}(S)\bigr]_{\psi}\\
&=H^{\sigma-2m,(\sigma-2m)/(2b);\varphi}_{+}(\Omega)\oplus
\bigoplus_{j=1}^{m}H^{\sigma-m_j-1/2,(\sigma-m_j-1/2)/(2b);\varphi}_{+}(S)\\
&=\mathcal{H}^{\sigma-2m,(\sigma-2m)/(2b);\varphi}_{+}(\Omega,S).
\end{align*}
These equalities of spaces hold up to equivalence of norms. Thus, the isomorphism \eqref{9f8.2} becomes \eqref{9f4.1}. This isomorphism is an extension by continuity of the mapping \eqref{9f2.4} because the set $C^{\infty}_{+}(\overline{\Omega})$ is dense in the space $H^{\sigma,\sigma/(2b);\varphi}_{+}(\Omega)$.
\end{proof}

\begin{proof}[Proof of Theorem $\ref{9th4.3}$.]
We will first prove that, under the conditions \eqref{9f4.5} and \eqref{9f4.6} of this theorem, the implication
\begin{equation}\label{9f8.3}
u\in H^{\sigma-\lambda,(\sigma-\lambda)/(2b);\varphi}_{+,\mathrm{loc}}
(\omega,\pi_1)\;\Rightarrow\;u\in H^{\sigma-\lambda+1,(\sigma-\lambda+1)/(2b);\varphi}_{+,\mathrm{loc}}
(\omega,\pi_1)
\end{equation}
holds for each integer $\lambda\geq1$ subject to $\sigma-\lambda+1>\sigma_{0}$.

We arbitrarily choose a function $\chi\in C^\infty(\overline\Omega)$ with $\mbox{supp}\,\chi\subseteq\omega\cup\pi_1$. For $\chi$ there exists a function $\eta\in C^\infty(\overline\Omega)$ such that $\mbox{supp}\,\eta\subseteq\omega\cup\pi_1$ and
$\eta=1$ in a neighbourhood of $\mbox{supp}\,\chi$. Interchanging each of the differential operators $A$ and $B_{j}$ with the operator of the multiplication by $\chi$, we can write
\begin{equation}\label{9f8.4}
\begin{aligned}
(A,B)(\chi u)&=(A,B)(\chi\eta u)=\chi\,(A,B)(\eta u)+ (A',B')(\eta u)\\
&=\chi\,(A,B)u+(A',B')(\eta u)=\chi\,(f,g_{1},...,g_{m})+(A',B')(\eta u).
\end{aligned}
\end{equation}
Here, $(A',B'):=(A',B'_{1},\ldots,B'_{m})$ is a differential operator
with components
\begin{equation}\label{9f8.5}
A'(x,t,D_x,\partial_t)=\sum_{|\alpha|+2b\beta\leq 2m-1}a^{\alpha,\beta}_{1}(x,t)\,D^\alpha_x\partial^\beta_t
\end{equation}
and
\begin{equation}\label{9f8.6}
B_{j}'(x,t,D_x,\partial_t)=\sum_{|\alpha|+2b\beta\leq m_j-1}
b_{j,1}^{\alpha,\beta}(x,t)\,D^\alpha_x\partial^\beta_t,\quad j=1,\ldots,m,
\end{equation}
where all $a^{\alpha,\beta}_{1}\in C^{\infty}(\overline{\Omega})$ and $b_{j,1}^{\alpha,\beta}\in C^{\infty}(\overline{S})$. This operator acts continuously between the spaces
\begin{equation}\label{9f8.7}
(A',B'):\,H^{s,s/(2b);\varphi}_{+}(\Omega)\rightarrow
\mathcal{H}^{s+1-2m,(s+1-2m)/(2b);\varphi}_{+}(\Omega,S)
\end{equation}
for every $s>\sigma_{0}-1$. In the case where $\varphi\equiv1$ and where the second superscripts are not half-integer, this follows directly from \eqref{9f8.5}, \eqref{9f8.6}, Lemma~\ref{9lem5.1} and the known properties of the anisotropic Sobolev space $H^{s,s/(2b)}(\Omega)$ (see, e.g., \cite[Chapter~I, Lemma~4, and Chapter~II, Theorems~3 and~7]{Slobodeckii58}). The boundedness of the operator \eqref{9f8.7} in the general situation is plainly deduced from this case with the help of the interpolation Lemma~\ref{9lem7.3}.

By the conditions \eqref{9f4.5} and \eqref{9f4.6}, we obtain the inclusion
$$
\chi\,(f,g_{1},...,g_{m})
\in\mathcal{H}^{\sigma-2m,(\sigma-2m)/(2b);\varphi}_{+}(\Omega,S).
$$
Besides, according to \eqref{9f8.7} with $s:=\sigma-\lambda$, we have the implication
\begin{align*}
u&\in H^{\sigma-\lambda,(\sigma-\lambda)/(2b);\varphi}_{+,\mathrm{loc}}
(\omega,\pi_{1})\\
&\Rightarrow\;(A',B')(\eta u)\in
\mathcal{H}^{\sigma-\lambda+1-2m,(\sigma-\lambda+1-2m)/(2b);\varphi}_{+}
(\Omega,S).
\end{align*}
Hence, using \eqref{9f8.4} and Corollary~\ref{9cor4.2}, we can write
\begin{align*}
u&\in H^{\sigma-\lambda,(\sigma-\lambda)/(2b);\varphi}_{+,\mathrm{loc}}
(\omega,\pi_{1})\\
&\Rightarrow\;(A,B)(\chi u)\in
\mathcal{H}^{\sigma-\lambda+1-2m,(\sigma-\lambda+1-2m)/(2b);\varphi}_{+}
(\Omega,S)\\
&\Rightarrow\;\chi u\in
H^{\sigma-\lambda+1,(\sigma-\lambda+1)/(2b);\varphi}_{+}(\Omega).
\end{align*}
Note that Corollary~\ref{9cor4.2} is applicable here because $\chi u\in H^{\sigma_0,\sigma_0/(2b)}_{+}(\Omega)$ by the condition of the theorem and because $\sigma-\lambda+1>\sigma_{0}$. Thus, owing to our choice of $\chi$, we have proved the implication \eqref{9f8.3}.

Let us use this implication in our proof of the inclusion $u\in H^{\sigma,\sigma/(2b);\varphi}_{+,\mathrm{loc}}(\omega,\pi_1)$. We will separately examine the case of $\sigma\notin\mathbb{Z}$ and the case of $\sigma\in\mathbb{Z}$.

Consider first the case of $\sigma\notin\mathbb{Z}$. In this case, there exists an integer $\lambda_{0}\geq1$ such that
\begin{equation}\label{9f8.8}
\sigma-\lambda_{0}<\sigma_{0}<\sigma-\lambda_{0}+1.
\end{equation}
Using the implication \eqref{9f8.3} successively with $\lambda:=\lambda_{0}$,
$\lambda:=\lambda_{0}-1$,..., and $\lambda:=1$, we deduce the desired inclusion in the following way:
\begin{align*}
u&\in H^{\sigma_0,\sigma_0/(2b)}_{+}(\Omega)\subset
H^{\sigma-\lambda_{0},(\sigma-\lambda_{0})/(2b);\varphi}_{+}(\Omega)\\
&\Rightarrow\;u\in
H^{\sigma-\lambda_{0}+1,(\sigma-\lambda_{0}+1)/(2b);\varphi}_{+}(\Omega)\\
&\Rightarrow\;\ldots\;\Rightarrow\;u\in H^{\sigma,\sigma/(2b);\varphi}_{+}(\Omega).
\end{align*}
Note that $u\in H^{\sigma_0,\sigma_0/(2b)}_{+}(\Omega)$ by the condition of the theorem.

Consider now the case of $\sigma\in\mathbb{Z}$. In this case, there is no integer $\lambda_{0}$ that satisfies~\eqref{9f8.8}. Nevertheless, since $\sigma-1/2\notin\mathbb{Z}$ and $\sigma-1/2>\sigma_{0}$, the inclusion
\begin{equation*}
u\in H^{\sigma-1/2,(\sigma-1/2)/(2b);\varphi}_{+}(\Omega)
\end{equation*}
holds true as we have proved in the previous paragraph. Hence,
using the implication \eqref{9f8.3} with $\lambda:=1$, we deduce the desired inclusion; namely:
\begin{align*}
u\in H^{\sigma-1/2,(\sigma-1/2)/(2b);\varphi}_{+}(\Omega)\subset
H^{\sigma-1,(\sigma-1)/(2b);\varphi}_{+}(\Omega)
\;\Rightarrow\;u\in H^{\sigma,\sigma/(2b);\varphi}_{+}(\Omega).
\end{align*}
\end{proof}

To deduce Theorem~\ref{9th4.4} from Theorem~\ref{9th4.3} and to justify Remark~\ref{9rem4.5}, we use a version of H\"ormander's embedding theorem \cite[Theorem~2.2.7]{Hormander63}.

\begin{lemma}\label{9lem8.1}
Let $p\in\mathbb{Z}$, $p\geq0$, $s:=p+b+n/2$, and $\varphi\in\mathcal{M}$. The following two assertions are true:
\begin{itemize}
\item[(i)] If $\varphi$ satisfies \eqref{9f4.7}, then every function $w\in H^{s,s/(2b);\varphi}(\mathbb{R}^{n+1})$ has the following property: all its generalized partial derivatives $D_{x}^{\alpha}\partial_{t}^{\beta}w(x,t)$ with $0\leq|\alpha|+2b\beta\leq p$ are continuous on $\mathbb{R}^{n+1}$.
\item[(ii)] Let $V$ be a nonempty open subset of $\mathbb{R}^{n+1}$, and let an integer $k$ be such that $1\leq k\leq n$. If every function $w\in H^{s,s/(2b);\varphi}(\mathbb{R}^{n+1})$ with $\mathrm{supp}\,w\subset V$ satisfies the condition $\partial_{k}^{j}w\in C(\mathbb{R}^{n+1})$ for each $j\in\mathbb{Z}$ with $0\leq j\leq p$, then $\varphi$ meets \eqref{9f4.7}. Here, $\partial_{k}^{j}w$ denotes the generalized partial derivative $(\partial^{k}w)/\partial{x_{k}^{j}}$ of the function $w=w(x_{1},\ldots,x_{n},t)$.
\end{itemize}
\end{lemma}

\begin{proof} (i) Given a function $w\in H^{s,s/(2b);\varphi}(\mathbb{R}^{n+1})$, we consider its arbitrary partial derivative $w_{\alpha,\beta}(x,t):=D_{x}^{\alpha}\partial_{t}^{\beta}w(x,t)$ with $0\leq|\alpha|+2b\beta\leq p$. The condition
\begin{equation}\label{9f8.9}
\int\limits_{\mathbb{R}^{n}}\int\limits_{\mathbb{R}}
\frac{|\xi^{\alpha}|^{2}\,|\eta|^{2\beta}\,d\xi\,d\eta}
{r_{\gamma}^{2s}(\xi,\eta)\,\varphi^{2}(r_{\gamma}(\xi,\eta))}<\infty
\end{equation}
implies the inclusion $w_{\alpha,\beta}\in C(\mathbb{R}^{n+1})$; here,  $\gamma:=1/(2b)$. Indeed, by the Schwarz inequality, we can write
\begin{equation*}
\int\limits_{\mathbb{R}^{n}}\int\limits_{\mathbb{R}}
|\widetilde{w_{\alpha,\beta}}(\xi,\eta)|\,d\xi\,d\eta\leq
\|w\|_{H^{s,s/(2b);\varphi}(\mathbb{R}^{n+1})}
\int\limits_{\mathbb{R}^{n}}\int\limits_{\mathbb{R}}
\frac{|\xi^{\alpha}|^{2}\,|\eta|^{2\beta}\,d\xi\,d\eta}
{r_{\gamma}^{2s}(\xi,\eta)\,\varphi^{2}(r_{\gamma}(\xi,\eta))}.
\end{equation*}
Hence, if the condition \eqref{9f8.9} is fulfilled, then the function $\widetilde{w_{\alpha,\beta}}$ is integrable over $\mathbb{R}^{n+1}$ and therefore its inverse Fourier transform $w_{\alpha,\beta}$ is continuous on $C(\mathbb{R}^{n+1})$.

Let us show that \eqref{9f4.7} implies this condition. In the multiple integral written in \eqref{9f8.9}, we change the variable $\eta=\eta_1^{2b}$, then pass to the spherical coordinates $(\varrho,\theta_{1},\ldots,\theta_{n})$ with $\varrho=(|\xi|^2+\eta_1^2)^{1/2}$, and finally change the variable $r=(1+\varrho^2)^{1/2}$. So, we write the following:
\begin{align*}
\int\limits_{\mathbb{R}^{n}}&\int\limits_{\mathbb{R}}
\frac{|\xi^{\alpha}|^{2}\,|\eta|^{2\beta}d\xi\,d\eta}
{r_{\gamma}^{2s}(\xi,\eta)\,\varphi^{2}(r_{\gamma}(\xi,\eta))}=
2^{n+1}\int\limits_{0}^{\infty}\dots
\int\limits_{0}^{\infty}\int\limits_{0}^{\infty}
\frac{|\xi^{\alpha}|^{2}\,\eta^{2\beta}d\xi\,d\eta}
{r_{\gamma}^{2s}(\xi,\eta)\,\varphi^{2}(r_{\gamma}(\xi,\eta))}\\
&=2^{n+1}\int\limits_{0}^{\infty}\dots
\int\limits_{0}^{\infty}\int\limits_{0}^{\infty}
\frac{2b\,|\xi^{\alpha}|^{2}\,\eta_1^{4b\beta+2b-1}\,d\xi\,d\eta_1}
{(1+|\xi|^2+\eta_1^2)^{s}
\varphi^2\bigl(\sqrt{1+|\xi|^2+\eta_1^2}\,\bigr)}\\
&=c_{\alpha,\beta}\int\limits_{0}^{\infty}
\frac{\varrho^{2|\alpha|+4b\beta+2b-1}\,\varrho^{n}\,d\varrho}
{(1+\varrho^2)^{s}\varphi^2(\sqrt{1+\varrho^2}\,)}=
c_{\alpha,\beta}\int\limits_{1}^{\infty}
\frac{(r^2-1)^{|\alpha|+2b\beta+b-1/2+n/2}\,r\,dr}{r^{2s}\,\varphi^2(r)\,
(r^2-1)^{1/2}}\\
&=c_{\alpha,\beta}\int\limits_{1}^{\infty}\;
\frac{(r^2-1)^{s-1-\delta(\alpha,\beta)}}{r^{2s-1}\,\varphi^2(r)}\,dr;
\end{align*}
here, $c_{\alpha,\beta}$ is a certain positive number, and
\begin{equation*}
\delta(\alpha,\beta):=p-|\alpha|-2b\beta\in[0,p].
\end{equation*}
Thus,
\begin{equation}\label{9f8.10}
\int\limits_{\mathbb{R}^{n}}\int\limits_{\mathbb{R}}
\frac{|\xi^{\alpha}|^{2}\,|\eta|^{2\beta}\,d\xi\,d\eta}
{r_{\gamma}^{2s}(\xi,\eta)\,\varphi^{2}(r_{\gamma}(\xi,\eta))}=
c_{\alpha,\beta}\int\limits_{1}^{\infty}\;
\frac{(r^2-1)^{s-1-\delta(\alpha,\beta)}}{r^{2s-1}\,\varphi^2(r)}\,dr.
\end{equation}
Note that
\begin{equation}\label{9f8.11}
\eqref{9f4.7}\;\;\Leftrightarrow\;\;\int\limits_{1}^{\infty}\;
\frac{(r^2-1)^{s-1}}{r^{2s-1}\,\varphi^2(r)}\,dr
<\infty
\end{equation}
Therefore \eqref{9f4.7} implies \eqref{9f8.9}. Assertion (i) is proved.

(ii) Let the assumption made in assertion (ii) be fulfilled. Following H\"ormander's proof of the necessity part of his embedding theorem \cite[Theorem 2.2.7]{Hormander63}, we conclude that
\begin{equation*}
\int\limits_{\mathbb{R}^{n}}\int\limits_{\mathbb{R}}
\frac{|\xi_{k}^{p}|^{2}\,d\xi\,d\eta}
{r_{\gamma}^{2s}(\xi,\eta)\,\varphi^{2}(r_{\gamma}(\xi,\eta))}<\infty.
\end{equation*}
This property together with \eqref{9f8.10} gives
\begin{equation}\label{9f8.12}
\int\limits_{1}^{\infty}\;\frac{(r^2-1)^{s-1}}
{r^{2s-1}\,\varphi^2(r)}\,dr<\infty.
\end{equation}
Here, we use \eqref{9f8.10} in the case where $\alpha_{k}=p$, $\alpha_{q}=0$ whenever $q\neq k$, and $\beta=0$, then $\delta(\alpha,\beta)=0$. Now \eqref{9f4.7} follows from \eqref{9f8.11} and \eqref{9f8.12}. Assertion (ii) is proved.
\end{proof}

\begin{proof}[Proof of Theorem $\ref{9th4.4}$.]
We arbitrarily choose a point $M\in\omega\cup\pi_1$. Let a function $\chi\in C^\infty(\overline\Omega)$ satisfy the following conditions:
$\mbox{supp}\,\chi\subseteq\omega\cup\pi_1$, and $\chi=1$ in a certain neighbourhood $V(M)\subseteq\overline{\Omega}$ of $M$. According to Theorem~\ref{9th4.3} we have the inclusion $\chi u\in H^{\sigma,\sigma/(2b);\varphi}_{+}(\Omega)$, with $\sigma=p+b+n/2$. Hence, there exists a function $w\in H^{\sigma,\sigma/(2b);\varphi}(\mathbb{R}^{n+1})$ such that $w=\chi
u=u$ on the set $V(M)$. According to Lemma~\ref{9lem8.1}(i), each generalized partial derivative $D_{x}^{\alpha}\partial_{t}^{\beta}w(x,t)$, with $0\leq|\alpha|+2b\beta\leq p$, is continuous on $\mathbb{R}^{n+1}$. Thus, the derivative $D_{x}^{\alpha}\partial_{t}u(x,t)$ is continuous in the neighbourhood $V(M)$ of $M$. Since the point $\nobreak{M\in\omega\cup\pi_1}$ is arbitrarily chosen, this derivative is continuous on $\omega\cup\pi_1$.
\end{proof}

At the end of this section, we will justify Remark~\ref{9rem4.5}. Let $\varphi\in\mathcal{M}$, and let an integer $p\geq0$ be subject to  $p+b+n/2>\sigma_{0}$. Assume that every solution $u\in\nobreak H^{\sigma_0,\sigma_0/(2b)}_{+}(\Omega)$ to the problem \eqref{9f2.1}--\eqref{9f2.3} whose right-hand sides satisfy conditions
\eqref{9f4.5} and \eqref{9f4.6} for $\sigma:=p+b+n/2$ complies with the conclusion of Theorem~\ref{9th4.4}. Then, for arbitrary function $u\in H^{\sigma,\sigma/(2b);\varphi}_{+}(\Omega)$, we put
\begin{equation*}
(f,g_{1},...,g_{m}):=(A,B)u\in
\mathcal{H}^{\sigma-2m,(\sigma-2m)/(2b);\varphi}_{+}(\Omega,S)
\end{equation*}
and can assert that $u$ meets the conclusion of Theorem~\ref{9th4.4}. Thus, if the function $u$ belongs to $H^{\sigma,\sigma/(2b);\varphi}_{+}(\Omega)$, then all its generalized derivatives $D_{x}^{\alpha}\partial_{t}^{\beta}u(x,t)$ with $|\alpha|+2b\beta\leq p$ are continuous on $\omega\cup\pi_1$. Specifically, each derivative $\partial_{1}^{j}u$ with $0\leq j\leq p$ is continuous on $\omega\cup\pi_1$.

Let $V$ be a nonempty open subset of $\mathbb{R}^{n+1}$ such that $\overline{V}\subset\omega$. We arbitrarily choose a function $w\in H^{\sigma,\sigma/(2b);\varphi}(\mathbb{R}^{n+1})$ such that $\mathrm{supp}\,w\subset V$, and we put
\begin{equation*}
u:=w\!\upharpoonright\!\Omega\in H^{\sigma,\sigma/(2b);\varphi}_{+}(\Omega).
\end{equation*}
The function $w$ and its generalized derivatives $\partial_{1}^{j}w$ with $1\leq j\leq p$ are continuous on $\mathbb{R}^{n+1}$ due to the property of $u$ deduced in the previous paragraph under the assumption made. Hence, $\varphi$ satisfies \eqref{9f4.7} due to Lemma~\ref{9lem8.1}(ii). Thus, we have justified Remark~\ref{9rem4.5}.

\end{document}